\documentclass[12pt]{amsart}
\usepackage[latin1]{inputenc}
\usepackage{amsbsy}
\usepackage{amscd} 
\usepackage{amsfonts}
\usepackage{amsgen} 
\usepackage{amsmath}
\usepackage{amsopn} 
\usepackage{amssymb}
\usepackage{amstext}
\usepackage{amsthm} 
\usepackage{amsxtra}
\usepackage{array}
\usepackage[all]{xy}
\usepackage{epsfig}
\usepackage{extarrows}
\usepackage{dsfont}
\usepackage{float}
\usepackage{hhline}
\usepackage{longtable}
\usepackage{mathtools}
\usepackage{paralist}
\usepackage{rotating}
\usepackage{tikz}

\usepackage{calc}
\newcounter{PartitionDepth}
\newcounter{PartitionLength}
\newsavebox{\boxpaarpart}
   \savebox{\boxpaarpart}
   { \begin{picture}(0.7,0.5)
     \put(0,0){\line(0,1){0.4}}
     \put(0.5,0){\line(0,1){0.4}}
     \put(0,0.4){\line(1,0){0.5}}
     \end{picture}}
\newsavebox{\boxbaarpart}
   \savebox{\boxbaarpart}
   { \begin{picture}(0.7,0.5)
     \put(0,0){\line(0,1){0.4}}
     \put(0.5,0){\line(0,1){0.4}}
     \put(0,0){\line(1,0){0.5}}
     \end{picture}}
\newsavebox{\boxdreipart}
   \savebox{\boxdreipart}
   { \begin{picture}(1,0.5)
     \put(0,0){\line(0,1){0.4}}
     \put(0.4,0){\line(0,1){0.4}}
     \put(0.8,0){\line(0,1){0.4}}
     \put(0,0.4){\line(1,0){0.8}}
     \end{picture}}
\newsavebox{\boxvierpart}
   \savebox{\boxvierpart}
   { \begin{picture}(1.4,0.5)
     \put(0,0){\line(0,1){0.4}}
     \put(0.4,0){\line(0,1){0.4}}
     \put(0.8,0){\line(0,1){0.4}}
     \put(1.2,0){\line(0,1){0.4}}
     \put(0,0.4){\line(1,0){1.2}}
     \end{picture}}
\newsavebox{\boxvierpartrot}
   \savebox{\boxvierpartrot}
   { \begin{picture}(0.5,0.8)
     \put(0,-0.2){\line(0,1){0.3}}
     \put(0.4,-0.2){\line(0,1){0.3}}
     \put(0,0.1){\line(1,0){0.4}}
     \put(0,0.4){\line(0,1){0.3}}
     \put(0.4,0.4){\line(0,1){0.3}}
     \put(0,0.4){\line(1,0){0.4}}
     \put(0.2,0.1){\line(0,1){0.3}}
     \end{picture}}
\newsavebox{\boxvierpartrotdrei}
   \savebox{\boxvierpartrotdrei}
   { \begin{picture}(1,0.8)
     \put(0,-0.2){\line(0,1){0.3}}
     \put(0.4,-0.2){\line(0,1){0.8}}
     \put(0.8,-0.2){\line(0,1){0.3}}
     \put(0,0.1){\line(1,0){0.8}}
     \end{picture}}
\newsavebox{\boxcrosspart}
   \savebox{\boxcrosspart}
   { \begin{picture}(0.5,0.8)
     \put(0,-0.2){\line(1,2){0.4}}
     \put(0.4,-0.2){\line(-1,2){0.4}}
     \end{picture}}
\newsavebox{\boxhalflibpart}
   \savebox{\boxhalflibpart}
   { \begin{picture}(1,0.8)
     \put(0,-0.2){\line(1,1){0.8}}
     \put(0.8,-0.2){\line(-1,1){0.8}}
     \put(0.4,-0.2){\line(0,1){0.8}}
     \end{picture}}
\newsavebox{\boxpositioner} 
   \savebox{\boxpositioner}
   { \begin{picture}(1.4,0.5)
     \put(0,0){\line(0,1){0.3}}
     \put(0.4,0){\line(0,1){0.6}}
     \put(0.8,0){\line(0,1){0.3}}
     \put(1.2,0){\line(0,1){0.6}}
     \put(0.4,0.6){\line(1,0){0.8}}
     \end{picture}}
\newsavebox{\boxfatcross} 
   \savebox{\boxfatcross}
   { \begin{picture}(1.4,1.3)
     \put(0,-0.2){\line(0,1){0.3}}
     \put(0.4,-0.2){\line(0,1){0.3}}
     \put(0.8,-0.2){\line(0,1){0.3}}
     \put(1.2,-0.2){\line(0,1){0.3}}
     \put(0,0.1){\line(1,0){0.4}}
     \put(0.8,0.1){\line(1,0){0.4}}
     \put(0,0.9){\line(0,1){0.3}}
     \put(0.4,0.9){\line(0,1){0.3}}
     \put(0.8,0.9){\line(0,1){0.3}}
     \put(1.2,0.9){\line(0,1){0.3}}
     \put(0,0.9){\line(1,0){0.4}}
     \put(0.8,0.9){\line(1,0){0.4}}
     \put(0.2,0.1){\line(1,1){0.8}}
     \put(0.2,0.9){\line(1,-1){0.8}}
     \end{picture}}
\newsavebox{\boxprimarypart} 
   \savebox{\boxprimarypart}
   { \begin{picture}(1,1.3)
     \put(0,-0.2){\line(0,1){0.3}}
     \put(0.4,-0.2){\line(0,1){0.3}}
     \put(0.8,-0.2){\line(0,1){0.3}}
     \put(0.4,0.1){\line(1,0){0.4}}
     \put(0,0.9){\line(0,1){0.3}}
     \put(0.4,0.9){\line(0,1){0.3}}
     \put(0.8,0.9){\line(0,1){0.3}}
     \put(0,0.9){\line(1,0){0.4}}
     \put(0,0.1){\line(1,1){0.8}}
     \put(0.2,0.9){\line(1,-2){0.4}}
     \end{picture}}

\newsavebox{\boxidpartww}
   \savebox{\boxidpartww}
   { \begin{picture}(0.5,1.2)
     \put(0.2,0.15){\line(0,1){0.7}}
     \put(0,-0.2){$\circ$}
     \put(0,0.8){$\circ$}
     \end{picture}}
\newsavebox{\boxidpartbw}
   \savebox{\boxidpartbw}
   { \begin{picture}(0.5,1.2)
     \put(0.2,0.15){\line(0,1){0.7}}
     \put(0,-0.2){$\circ$}
     \put(0,0.8){$\bullet$}
     \end{picture}}
\newsavebox{\boxidpartwb}
   \savebox{\boxidpartwb}
   { \begin{picture}(0.5,1.2)
     \put(0.2,0.15){\line(0,1){0.7}}
     \put(0,-0.2){$\bullet$}
     \put(0,0.8){$\circ$}
     \end{picture}}
\newsavebox{\boxidpartbb}
   \savebox{\boxidpartbb}
   { \begin{picture}(0.5,1.2)
     \put(0.2,0.15){\line(0,1){0.7}}
     \put(0,-0.2){$\bullet$}
     \put(0,0.8){$\bullet$}
     \end{picture}}

\newsavebox{\boxidpartsingletonbb}
   \savebox{\boxidpartsingletonbb}
   { \begin{picture}(0.5,1.2)
     \put(0.2,0.15){\line(0,1){0.2}}
     \put(0.2,0.65){\line(0,1){0.2}}
     \put(0,-0.2){$\bullet$}
     \put(0,0.8){$\bullet$}
     \end{picture}}
\newsavebox{\boxidpartsingletonww}
   \savebox{\boxidpartsingletonww}
   { \begin{picture}(0.5,1.2)
     \put(0.2,0.15){\line(0,1){0.2}}
     \put(0.2,0.65){\line(0,1){0.2}}
     \put(0,-0.2){$\circ$}
     \put(0,0.8){$\circ$}
     \end{picture}}
     
\newsavebox{\boxpaarpartbb}
   \savebox{\boxpaarpartbb}
   { \begin{picture}(1,1)
     \put(0.2,0.2){\line(0,1){0.4}}
     \put(0.7,0.2){\line(0,1){0.4}}
     \put(0.2,0.6){\line(1,0){0.5}}
     \put(0.5,-0.2){$\bullet$}
     \put(0,-0.2){$\bullet$}
     \end{picture}}
\newsavebox{\boxpaarpartww}
   \savebox{\boxpaarpartww}
   { \begin{picture}(1,1)
     \put(0.2,0.2){\line(0,1){0.4}}
     \put(0.7,0.2){\line(0,1){0.4}}
     \put(0.2,0.6){\line(1,0){0.5}}
     \put(0.5,-0.2){$\circ$}
     \put(0,-0.2){$\circ$}
     \end{picture}}
\newsavebox{\boxpaarpartbw}
   \savebox{\boxpaarpartbw}
   { \begin{picture}(1,1)
     \put(0.2,0.2){\line(0,1){0.4}}
     \put(0.7,0.2){\line(0,1){0.4}}
     \put(0.2,0.6){\line(1,0){0.5}}
     \put(0.5,-0.2){$\circ$}
     \put(0,-0.2){$\bullet$}
     \end{picture}}
\newsavebox{\boxpaarpartwb}
   \savebox{\boxpaarpartwb}
   { \begin{picture}(1,1)
     \put(0.2,0.2){\line(0,1){0.4}}
     \put(0.7,0.2){\line(0,1){0.4}}
     \put(0.2,0.6){\line(1,0){0.5}}
     \put(0.5,-0.2){$\bullet$}
     \put(0,-0.2){$\circ$}
     \end{picture}}
\newsavebox{\boxbaarpartbb}
   \savebox{\boxbaarpartbb}
   { \begin{picture}(1,1)
     \put(0.2,-0.1){\line(0,1){0.4}}
     \put(0.7,-0.1){\line(0,1){0.4}}
     \put(0.2,-0.1){\line(1,0){0.5}}
     \put(0.5,0.3){$\bullet$}
     \put(0,0.3){$\bullet$}
     \end{picture}}
\newsavebox{\boxbaarpartww}
   \savebox{\boxbaarpartww}
   { \begin{picture}(1,1)
     \put(0.2,-0.1){\line(0,1){0.4}}
     \put(0.7,-0.1){\line(0,1){0.4}}
     \put(0.2,-0.1){\line(1,0){0.5}}
     \put(0.5,0.3){$\circ$}
     \put(0,0.3){$\circ$}
     \end{picture}}
\newsavebox{\boxbaarpartbw}
   \savebox{\boxbaarpartbw}
   { \begin{picture}(1,1)
     \put(0.2,-0.1){\line(0,1){0.4}}
     \put(0.7,-0.1){\line(0,1){0.4}}
     \put(0.2,-0.1){\line(1,0){0.5}}
     \put(0.5,0.3){$\circ$}
     \put(0,0.3){$\bullet$}
     \end{picture}}
\newsavebox{\boxbaarpartwb}
   \savebox{\boxbaarpartwb}
   { \begin{picture}(1,1)
     \put(0.2,-0.1){\line(0,1){0.4}}
     \put(0.7,-0.1){\line(0,1){0.4}}
     \put(0.2,-0.1){\line(1,0){0.5}}
     \put(0.5,0.3){$\bullet$}
     \put(0,0.3){$\circ$}
     \end{picture}}
\newsavebox{\boxpaarbaarpartwwww}
   \savebox{\boxpaarbaarpartwwww}
   { \begin{picture}(1,1.2)
     \put(0.2,0.15){\line(0,1){0.2}}
     \put(0.2,0.65){\line(0,1){0.2}}
     \put(0.2,0.65){\line(1,0){0.5}}
     \put(0.2,0.35){\line(1,0){0.5}}
     \put(0.7,0.15){\line(0,1){0.2}}
     \put(0.7,0.65){\line(0,1){0.2}}
     \put(0,0.8){$\circ$}
     \put(0.5,0.8){$\circ$}
     \put(0,-0.2){$\circ$}
     \put(0.5,-0.2){$\circ$}
     \end{picture}}     

\newsavebox{\boxsingletonw}
   \savebox{\boxsingletonw}
   { \begin{picture}(0.5, 1)
     \put(0,0.3){$\uparrow$}
     \put(0,-0.2){$\circ$}
     \end{picture}}
\newsavebox{\boxsingletonb}
   \savebox{\boxsingletonb}
   { \begin{picture}(0.5, 1)
     \put(0,0.3){$\uparrow$}
     \put(0,-0.2){$\bullet$}
     \end{picture}}
\newsavebox{\boxdownsingletonw}
   \savebox{\boxdownsingletonw}
   { \begin{picture}(0.5, 1)
     \put(0,0){$\downarrow$}
     \put(0,0.6){$\circ$}
     \end{picture}}
\newsavebox{\boxdownsingletonb}
   \savebox{\boxdownsingletonb}
   { \begin{picture}(0.5, 1)
     \put(0,0){$\downarrow$}
     \put(0,0.6){$\bullet$}
     \end{picture}}

\newsavebox{\boxvierpartwbwb}
   \savebox{\boxvierpartwbwb}
   { \begin{picture}(1.7,1)
     \put(0.2,0.2){\line(0,1){0.4}}
     \put(0.6,0.2){\line(0,1){0.4}}
     \put(1,0.2){\line(0,1){0.4}}
     \put(1.4,0.2){\line(0,1){0.4}}
     \put(0.2,0.6){\line(1,0){1.2}}
     \put(0,-0.2){$\circ$}
     \put(0.4,-0.2){$\bullet$}
     \put(0.8,-0.2){$\circ$}
     \put(1.2,-0.2){$\bullet$}
     \end{picture}}
\newsavebox{\boxvierpartbwbw}
   \savebox{\boxvierpartbwbw}
   { \begin{picture}(1.7,1)
     \put(0.2,0.2){\line(0,1){0.4}}
     \put(0.6,0.2){\line(0,1){0.4}}
     \put(1,0.2){\line(0,1){0.4}}
     \put(1.4,0.2){\line(0,1){0.4}}
     \put(0.2,0.6){\line(1,0){1.2}}
     \put(0,-0.2){$\bullet$}
     \put(0.4,-0.2){$\circ$}
     \put(0.8,-0.2){$\bullet$}
     \put(1.2,-0.2){$\circ$}
     \end{picture}}
\newsavebox{\boxvierpartwwbb}
   \savebox{\boxvierpartwwbb}
   { \begin{picture}(1.7,1)
     \put(0.2,0.2){\line(0,1){0.4}}
     \put(0.6,0.2){\line(0,1){0.4}}
     \put(1,0.2){\line(0,1){0.4}}
     \put(1.4,0.2){\line(0,1){0.4}}
     \put(0.2,0.6){\line(1,0){1.2}}
     \put(0,-0.2){$\circ$}
     \put(0.4,-0.2){$\circ$}
     \put(0.8,-0.2){$\bullet$}
     \put(1.2,-0.2){$\bullet$}
     \end{picture}}
\newsavebox{\boxvierpartwbbw}
   \savebox{\boxvierpartwbbw}
   { \begin{picture}(1.7,1)
     \put(0.2,0.2){\line(0,1){0.4}}
     \put(0.6,0.2){\line(0,1){0.4}}
     \put(1,0.2){\line(0,1){0.4}}
     \put(1.4,0.2){\line(0,1){0.4}}
     \put(0.2,0.6){\line(1,0){1.2}}
     \put(0,-0.2){$\circ$}
     \put(0.4,-0.2){$\bullet$}
     \put(0.8,-0.2){$\bullet$}
     \put(1.2,-0.2){$\circ$}
     \end{picture}}
\newsavebox{\boxvierpartbwwb}
   \savebox{\boxvierpartbwwb}
   { \begin{picture}(1.7,1)
     \put(0.2,0.2){\line(0,1){0.4}}
     \put(0.6,0.2){\line(0,1){0.4}}
     \put(1,0.2){\line(0,1){0.4}}
     \put(1.4,0.2){\line(0,1){0.4}}
     \put(0.2,0.6){\line(1,0){1.2}}
     \put(0,-0.2){$\bullet$}
     \put(0.4,-0.2){$\circ$}
     \put(0.8,-0.2){$\circ$}
     \put(1.2,-0.2){$\bullet$}
     \end{picture}}
\newsavebox{\boxvierpartrotwbwb}
   \savebox{\boxvierpartrotwbwb}
   { \begin{picture}(1,1.2)
     \put(0.2,0.15){\line(0,1){0.2}}
     \put(0.2,0.65){\line(0,1){0.2}}
     \put(0.2,0.65){\line(1,0){0.5}}
     \put(0.2,0.35){\line(1,0){0.5}}
     \put(0.45,0.35){\line(0,1){0.3}}
     \put(0.7,0.15){\line(0,1){0.2}}
     \put(0.7,0.65){\line(0,1){0.2}}
     \put(0,0.8){$\circ$}
     \put(0.5,0.8){$\bullet$}
     \put(0,-0.2){$\circ$}
     \put(0.5,-0.2){$\bullet$}
     \end{picture}}
\newsavebox{\boxvierpartrotbwwb}
   \savebox{\boxvierpartrotbwwb}
   { \begin{picture}(1,1.2)
     \put(0.2,0.15){\line(0,1){0.2}}
     \put(0.2,0.65){\line(0,1){0.2}}
     \put(0.2,0.65){\line(1,0){0.5}}
     \put(0.2,0.35){\line(1,0){0.5}}
     \put(0.45,0.35){\line(0,1){0.3}}
     \put(0.7,0.15){\line(0,1){0.2}}
     \put(0.7,0.65){\line(0,1){0.2}}
     \put(0,0.8){$\bullet$}
     \put(0.5,0.8){$\circ$}
     \put(0,-0.2){$\circ$}
     \put(0.5,-0.2){$\bullet$}
     \end{picture}}
\newsavebox{\boxvierpartrotbwbw}
   \savebox{\boxvierpartrotbwbw}
   { \begin{picture}(1,1.2)
     \put(0.2,0.15){\line(0,1){0.2}}
     \put(0.2,0.65){\line(0,1){0.2}}
     \put(0.2,0.65){\line(1,0){0.5}}
     \put(0.2,0.35){\line(1,0){0.5}}
     \put(0.45,0.35){\line(0,1){0.3}}
     \put(0.7,0.15){\line(0,1){0.2}}
     \put(0.7,0.65){\line(0,1){0.2}}
     \put(0,0.8){$\bullet$}
     \put(0.5,0.8){$\circ$}
     \put(0,-0.2){$\bullet$}
     \put(0.5,-0.2){$\circ$}
     \end{picture}}
\newsavebox{\boxvierpartrotwwww}
   \savebox{\boxvierpartrotwwww}
   { \begin{picture}(1,1.2)
     \put(0.2,0.15){\line(0,1){0.2}}
     \put(0.2,0.65){\line(0,1){0.2}}
     \put(0.2,0.65){\line(1,0){0.5}}
     \put(0.2,0.35){\line(1,0){0.5}}
     \put(0.45,0.35){\line(0,1){0.3}}
     \put(0.7,0.15){\line(0,1){0.2}}
     \put(0.7,0.65){\line(0,1){0.2}}
     \put(0,0.8){$\circ$}
     \put(0.5,0.8){$\circ$}
     \put(0,-0.2){$\circ$}
     \put(0.5,-0.2){$\circ$}
     \end{picture}}
\newsavebox{\boxvierpartrotbbbb}
   \savebox{\boxvierpartrotbbbb}
   { \begin{picture}(1,1.2)
     \put(0.2,0.15){\line(0,1){0.2}}
     \put(0.2,0.65){\line(0,1){0.2}}
     \put(0.2,0.65){\line(1,0){0.5}}
     \put(0.2,0.35){\line(1,0){0.5}}
     \put(0.45,0.35){\line(0,1){0.3}}
     \put(0.7,0.15){\line(0,1){0.2}}
     \put(0.7,0.65){\line(0,1){0.2}}
     \put(0,0.8){$\bullet$}
     \put(0.5,0.8){$\bullet$}
     \put(0,-0.2){$\bullet$}
     \put(0.5,-0.2){$\bullet$}
     \end{picture}}

\newsavebox{\boxdreipartwww}
   \savebox{\boxdreipartwww}
   { \begin{picture}(1.2,1)
     \put(0.2,0.2){\line(0,1){0.4}}
     \put(0.6,0.2){\line(0,1){0.4}}
     \put(1,0.2){\line(0,1){0.4}}
     \put(0.2,0.6){\line(1,0){0.8}}
     \put(0,-0.2){$\circ$}
     \put(0.4,-0.2){$\circ$}
     \put(0.8,-0.2){$\circ$}
     \end{picture}}
\newsavebox{\boxdreipartrotwww}
   \savebox{\boxdreipartrotwww}
   { \begin{picture}(1,1.2)
     \put(0.2,0.65){\line(0,1){0.2}}
     \put(0.2,0.65){\line(1,0){0.5}}
     \put(0.45,0.15){\line(0,1){0.5}}
     \put(0.7,0.65){\line(0,1){0.2}}
     \put(0,0.8){$\circ$}
     \put(0.5,0.8){$\circ$}
     \put(0.25,-0.2){$\circ$}
     \end{picture}}     
\newsavebox{\boxdowndreipartrotwww}
   \savebox{\boxdowndreipartrotwww}
   { \begin{picture}(1,1.2)
     \put(0.2,0.15){\line(0,1){0.2}}
     \put(0.2,0.35){\line(1,0){0.5}}
     \put(0.45,0.35){\line(0,1){0.5}}
     \put(0.7,0.15){\line(0,1){0.2}}
     \put(0,-0.2){$\circ$}
     \put(0.5,-0.2){$\circ$}
     \put(0.25,0.8){$\circ$}
     \end{picture}}          

\newsavebox{\boxsechspartwbwbwb}
   \savebox{\boxsechspartwbwbwb}
   { \begin{picture}(2.5,1)
     \put(0.2,0.2){\line(0,1){0.4}}
     \put(0.6,0.2){\line(0,1){0.4}}
     \put(1,0.2){\line(0,1){0.4}}
     \put(1.4,0.2){\line(0,1){0.4}}
     \put(1.8,0.2){\line(0,1){0.4}}
     \put(2.2,0.2){\line(0,1){0.4}}
     \put(0.2,0.6){\line(1,0){2}}
     \put(0,-0.2){$\circ$}
     \put(0.4,-0.2){$\bullet$}
     \put(0.8,-0.2){$\circ$}
     \put(1.2,-0.2){$\bullet$}
     \put(1.6,-0.2){$\circ$}
     \put(2.0,-0.2){$\bullet$}
     \end{picture}}
 
\newsavebox{\boxcrosspartwbbw}
   \savebox{\boxcrosspartwbbw}
   { \begin{picture}(1,1.4)
     \put(0.25,0.15){\line(1,2){0.4}}
     \put(0.65,0.15){\line(-1,2){0.4}}
     \put(0,0.9){$\circ$}
     \put(0.5,0.9){$\bullet$}
     \put(0,-0.2){$\bullet$}
     \put(0.5,-0.2){$\circ$}
     \end{picture}}
\newsavebox{\boxcrosspartbwwb}
   \savebox{\boxcrosspartbwwb}
   { \begin{picture}(1,1.4)
     \put(0.25,0.15){\line(1,2){0.4}}
     \put(0.65,0.15){\line(-1,2){0.4}}
     \put(0,0.9){$\bullet$}
     \put(0.5,0.9){$\circ$}
     \put(0,-0.2){$\circ$}
     \put(0.5,-0.2){$\bullet$}
     \end{picture}}
\newsavebox{\boxcrosspartwwww}
   \savebox{\boxcrosspartwwww}
   { \begin{picture}(1,1.4)
     \put(0.25,0.15){\line(1,2){0.4}}
     \put(0.65,0.15){\line(-1,2){0.4}}
     \put(0,0.9){$\circ$}
     \put(0.5,0.9){$\circ$}
     \put(0,-0.2){$\circ$}
     \put(0.5,-0.2){$\circ$}
     \end{picture}}
\newsavebox{\boxcrosspartbbbb}
   \savebox{\boxcrosspartbbbb}
   { \begin{picture}(1,1.4)
     \put(0.25,0.15){\line(1,2){0.4}}
     \put(0.65,0.15){\line(-1,2){0.4}}
     \put(0,0.9){$\bullet$}
     \put(0.5,0.9){$\bullet$}
     \put(0,-0.2){$\bullet$}
     \put(0.5,-0.2){$\bullet$}
     \end{picture}}

\newsavebox{\boxhalflibpartwwwwww}
   \savebox{\boxhalflibpartwwwwww}
   { \begin{picture}(1.5,1.4)
     \put(0.3,0.15){\line(1,1){0.8}}
     \put(1.1,0.15){\line(-1,1){0.8}}
     \put(0.7,0.15){\line(0,1){0.8}}     
     \put(0,0.9){$\circ$}
     \put(0.5,0.9){$\circ$}
     \put(1,0.9){$\circ$}
     \put(0,-0.2){$\circ$}
     \put(0.5,-0.2){$\circ$}
     \put(1,-0.2){$\circ$}     
     \end{picture}}

\newsavebox{\boxpositionerd} 
   \savebox{\boxpositionerd}
   { \begin{picture}(2.6,1.5)
     \put(0.9,0.2){\line(0,1){0.9}}
     \put(2.3,0.2){\line(0,1){0.9}}
     \put(0.9,1.1){\line(1,0){1.4}}
     \put(0.05,-0.05){$^\uparrow$}
     \put(0.2,0.4){\tiny$^{\otimes d}$}
     \put(1.35,-0.05){$^\uparrow$}
     \put(1.5,0.4){\tiny$^{\otimes d}$}
     \put(0,-0.2){$\circ$}
     \put(0.7,-0.2){$\circ$}
     \put(1.3,-0.2){$\bullet$}
     \put(2.1,-0.2){$\bullet$}
     \end{picture}}
\newsavebox{\boxpositionerl} 
   \savebox{\boxpositionerl}
   { \begin{picture}(2.6,1.5)
     \put(0.9,0.2){\line(0,1){0.9}}
     \put(2.3,0.2){\line(0,1){0.9}}
     \put(0.9,1.1){\line(1,0){1.4}}
     \put(0.05,-0.05){$^\uparrow$}
     \put(0.2,0.4){\tiny$^{\otimes l}$}
     \put(1.35,-0.05){$^\uparrow$}
     \put(1.5,0.4){\tiny$^{\otimes l}$}
     \put(0,-0.2){$\circ$}
     \put(0.7,-0.2){$\circ$}
     \put(1.3,-0.2){$\bullet$}
     \put(2.1,-0.2){$\bullet$}
     \end{picture}}
\newsavebox{\boxpositionerrpluseins} 
   \savebox{\boxpositionerrpluseins}
   { \begin{picture}(3.9,1.5)
     \put(1.6,0.2){\line(0,1){0.9}}
     \put(3.6,0.2){\line(0,1){0.9}}
     \put(1.6,1.1){\line(1,0){2}}
     \put(0.05,-0.05){$^\uparrow$}
     \put(0.2,0.4){\tiny$^{\otimes r+1}$}
     \put(2.05,-0.05){$^\uparrow$}
     \put(2.2,0.4){\tiny$^{\otimes r-1}$}
     \put(0,-0.2){$\circ$}
     \put(1.4,-0.2){$\bullet$}
     \put(2,-0.2){$\bullet$}
     \put(3.4,-0.2){$\bullet$}
     \end{picture}}
\newsavebox{\boxpositionerdt} 
   \savebox{\boxpositionerdt}
   { \begin{picture}(2.9,1.5)
     \put(1.2,0.2){\line(0,1){0.9}}
     \put(2.9,0.2){\line(0,1){0.9}}
     \put(1.2,1.1){\line(1,0){1.7}}
     \put(0.05,-0.05){$^\uparrow$}
     \put(0.2,0.4){\tiny$^{\otimes dt}$}
     \put(1.65,-0.05){$^\uparrow$}
     \put(1.8,0.4){\tiny$^{\otimes dt}$}
     \put(0,-0.2){$\circ$}
     \put(1,-0.2){$\circ$}
     \put(1.6,-0.2){$\bullet$}
     \put(2.7,-0.2){$\bullet$}
     \end{picture}}
\newsavebox{\boxpositioners} 
   \savebox{\boxpositioners}
   { \begin{picture}(2.6,1.5)
     \put(0.9,0.2){\line(0,1){0.9}}
     \put(2.3,0.2){\line(0,1){0.9}}
     \put(0.9,1.1){\line(1,0){1.4}}
     \put(0.05,-0.05){$^\uparrow$}
     \put(0.2,0.4){\tiny$^{\otimes s}$}
     \put(1.35,-0.05){$^\uparrow$}
     \put(1.5,0.4){\tiny$^{\otimes s}$}
     \put(0,-0.2){$\circ$}
     \put(0.7,-0.2){$\circ$}
     \put(1.3,-0.2){$\bullet$}
     \put(2.1,-0.2){$\bullet$}
     \end{picture}}
\newsavebox{\boxpositionerdinv} 
   \savebox{\boxpositionerdinv}
   { \begin{picture}(2.6,1.5)
     \put(0.9,0.2){\line(0,1){0.9}}
     \put(2.3,0.2){\line(0,1){0.9}}
     \put(0.9,1.1){\line(1,0){1.4}}
     \put(0.05,-0.05){$^\uparrow$}
     \put(0.2,0.4){\tiny$^{\otimes d}$}
     \put(1.35,-0.05){$^\uparrow$}
     \put(1.5,0.4){\tiny$^{\otimes d}$}
     \put(0,-0.2){$\circ$}
     \put(0.7,-0.2){$\bullet$}
     \put(1.3,-0.2){$\bullet$}
     \put(2.1,-0.2){$\circ$}
     \end{picture}}
\newsavebox{\boxpositionersinv} 
   \savebox{\boxpositionersinv}
   { \begin{picture}(2.6,1.5)
     \put(0.9,0.2){\line(0,1){0.9}}
     \put(2.3,0.2){\line(0,1){0.9}}
     \put(0.9,1.1){\line(1,0){1.4}}
     \put(0.05,-0.05){$^\uparrow$}
     \put(0.2,0.4){\tiny$^{\otimes s}$}
     \put(1.35,-0.05){$^\uparrow$}
     \put(1.5,0.4){\tiny$^{\otimes s}$}
     \put(0,-0.2){$\circ$}
     \put(0.7,-0.2){$\bullet$}
     \put(1.3,-0.2){$\bullet$}
     \put(2.1,-0.2){$\circ$}
     \end{picture}}
\newsavebox{\boxpositionerdpluszwei}
   \savebox{\boxpositionerdpluszwei}
   { \begin{picture}(3.4,1.5)
     \put(1.6,0.2){\line(0,1){0.9}}
     \put(3.0,0.2){\line(0,1){0.9}}
     \put(1.6,1.1){\line(1,0){1.4}}
     \put(0.05,-0.05){$^\uparrow$}
     \put(0.2,0.4){\tiny$^{\otimes d+2}$}
     \put(2.05,-0.05){$^\uparrow$}
     \put(2.2,0.4){\tiny$^{\otimes d}$}
     \put(0,-0.2){$\circ$}
     \put(1.4,-0.2){$\bullet$}
     \put(2,-0.2){$\bullet$}
     \put(2.8,-0.2){$\bullet$}
     \end{picture}}
\newsavebox{\boxpositionersminuszwei}
   \savebox{\boxpositionersminuszwei}
   { \begin{picture}(3.4,1.5)
     \put(1,0.2){\line(0,1){0.9}}
     \put(3.0,0.2){\line(0,1){0.9}}
     \put(1,1.1){\line(1,0){2}}
     \put(0.05,-0.05){$^\uparrow$}
     \put(0.2,0.4){\tiny$^{\otimes s}$}
     \put(1.45,-0.05){$^\uparrow$}
     \put(1.6,0.4){\tiny$^{\otimes s-2}$}
     \put(0,-0.2){$\circ$}
     \put(0.8,-0.2){$\bullet$}
     \put(1.4,-0.2){$\bullet$}
     \put(2.8,-0.2){$\bullet$}
     \end{picture}}
\newsavebox{\boxpositionerrnull}
   \savebox{\boxpositionerrnull}
   { \begin{picture}(3.8,1.5)
     \put(1.2,0.2){\line(0,1){0.9}}
     \put(3.4,0.2){\line(0,1){0.9}}
     \put(1.2,1.1){\line(1,0){2.2}}
     \put(0.05,-0.05){$^\uparrow$}
     \put(0.2,0.4){\tiny$^{\otimes r_0}$}
     \put(1.65,-0.05){$^\uparrow$}
     \put(1.8,0.4){\tiny$^{\otimes r_0-2}$}
     \put(0,-0.2){$\circ$}
     \put(1,-0.2){$\bullet$}
     \put(1.6,-0.2){$\bullet$}
     \put(3.2,-0.2){$\bullet$}
     \end{picture}}
\newsavebox{\boxpositionerwbwb}
   \savebox{\boxpositionerwbwb}
   { \begin{picture}(1.8,1)
     \put(0.6,0.2){\line(0,1){0.6}}
     \put(1.4,0.2){\line(0,1){0.6}}
     \put(0.6,0.8){\line(1,0){0.8}}
     \put(0.05,-0.05){$^\uparrow$}
     \put(0.85,-0.05){$^\uparrow$}
     \put(0,-0.2){$\circ$}
     \put(0.4,-0.2){$\bullet$}
     \put(0.8,-0.2){$\circ$}
     \put(1.2,-0.2){$\bullet$}
     \end{picture}}
\newsavebox{\boxpositionerwwbb}
   \savebox{\boxpositionerwwbb}
   { \begin{picture}(1.8,1)
     \put(0.6,0.2){\line(0,1){0.6}}
     \put(1.4,0.2){\line(0,1){0.6}}
     \put(0.6,0.8){\line(1,0){0.8}}
     \put(0.05,-0.05){$^\uparrow$}
     \put(0.85,-0.05){$^\uparrow$}
     \put(0,-0.2){$\circ$}
     \put(0.4,-0.2){$\circ$}
     \put(0.8,-0.2){$\bullet$}
     \put(1.2,-0.2){$\bullet$}
     \end{picture}}
\newsavebox{\boxpositionerwwww}
   \savebox{\boxpositionerwwww}
   { \begin{picture}(1.8,1)
     \put(0.6,0.2){\line(0,1){0.6}}
     \put(1.4,0.2){\line(0,1){0.6}}
     \put(0.6,0.8){\line(1,0){0.8}}
     \put(0.05,-0.05){$^\uparrow$}
     \put(0.85,-0.05){$^\uparrow$}
     \put(0,-0.2){$\circ$}
     \put(0.4,-0.2){$\circ$}
     \put(0.8,-0.2){$\circ$}
     \put(1.2,-0.2){$\circ$}
     \end{picture}}
\newsavebox{\boxpositionerrevwbwb}
   \savebox{\boxpositionerrevwbwb}
   { \begin{picture}(1.8,1)
     \put(0.2,0.2){\line(0,1){0.6}}
     \put(1.0,0.2){\line(0,1){0.6}}
     \put(0.2,0.8){\line(1,0){0.8}}
     \put(0.45,-0.05){$^\uparrow$}
     \put(1.25,-0.05){$^\uparrow$}
     \put(0,-0.2){$\circ$}
     \put(0.4,-0.2){$\bullet$}
     \put(0.8,-0.2){$\circ$}
     \put(1.2,-0.2){$\bullet$}
     \end{picture}}
\newsavebox{\boxpositionersalphaw} 
   \savebox{\boxpositionersalphaw}
   { \begin{picture}(2.6,1.5)
     \put(0.9,0.2){\line(0,1){0.9}}
     \put(2.3,0.2){\line(0,1){0.9}}
     \put(0.9,1.1){\line(1,0){1.4}}
     \put(0.05,-0.05){$^\uparrow$}
     \put(1.35,-0.05){$^\uparrow$}
     \put(0.2,0.4){\tiny$^{\otimes s}$}
     \put(1.5,0.4){\tiny$^{\otimes \alpha}$}
     \put(0,-0.2){$\circ$}
     \put(0.7,-0.2){$\circ$}
     \put(1.3,-0.2){$\circ$}
     \put(2.1,-0.2){$\bullet$}
     \end{picture}}
\newsavebox{\boxpositionersalphab} 
   \savebox{\boxpositionersalphab}
   { \begin{picture}(2.6,1.5)
     \put(0.9,0.2){\line(0,1){0.9}}
     \put(2.3,0.2){\line(0,1){0.9}}
     \put(0.9,1.1){\line(1,0){1.4}}
     \put(0.05,-0.05){$^\uparrow$}
     \put(1.35,-0.05){$^\uparrow$}
     \put(0.2,0.4){\tiny$^{\otimes s}$}
     \put(1.5,0.4){\tiny$^{\otimes \alpha}$}
     \put(0,-0.2){$\circ$}
     \put(0.7,-0.2){$\circ$}
     \put(1.3,-0.2){$\bullet$}
     \put(2.1,-0.2){$\bullet$}
     \end{picture}}

\newsavebox{\boxfatcrosswwww}
   \savebox{\boxfatcrosswwww}
   { \begin{picture}(2,1.6)
     \put(0.2,0.15){\line(0,1){0.2}}
     \put(0.2,0.95){\line(0,1){0.2}}
     \put(0.2,0.95){\line(1,0){0.5}}
     \put(0.2,0.35){\line(1,0){0.5}}
     \put(0.7,0.15){\line(0,1){0.2}}
     \put(0.7,0.95){\line(0,1){0.2}}
     \put(0,1.1){$\circ$}
     \put(0.5,1.1){$\circ$}
     \put(0,-0.2){$\circ$}
     \put(0.5,-0.2){$\circ$}
      \put(0.5,0.35){\line(2,1){1}}
      \put(0.5,0.85){\line(2,-1){1}}
     \put(1.2,0.15){\line(0,1){0.2}}
     \put(1.2,0.95){\line(0,1){0.2}}
     \put(1.2,0.95){\line(1,0){0.5}}
     \put(1.2,0.35){\line(1,0){0.5}}
     \put(1.7,0.15){\line(0,1){0.2}}
     \put(1.7,0.95){\line(0,1){0.2}}
     \put(1,1.1){$\circ$}
     \put(1.5,1.1){$\circ$}
     \put(1,-0.2){$\circ$}
     \put(1.5,-0.2){$\circ$}
     \end{picture}}
\newsavebox{\boxpairpositionerwwwwww}
   \savebox{\boxpairpositionerwwwwww}
   { \begin{picture}(2,1.6)
     \put(0.2,0.95){\line(0,1){0.2}}
     \put(0.2,0.95){\line(1,0){0.5}}
     \put(0.7,0.95){\line(0,1){0.2}}
     \put(0,1.1){$\circ$}
     \put(0.5,1.1){$\circ$}
     \put(0.25,-0.2){$\circ$}
       \put(0.45,0.1){\line(1,1){1}}
      \put(0.5,0.85){\line(2,-1){1}}
     \put(1.2,0.15){\line(0,1){0.2}}
     \put(1.2,0.35){\line(1,0){0.5}}
     \put(1.7,0.15){\line(0,1){0.2}}
     \put(1.25,1.1){$\circ$}
     \put(1,-0.2){$\circ$}
     \put(1.5,-0.2){$\circ$}
     \end{picture}}     

\newsavebox{\boxAspace}
   \savebox{\boxAspace}
   { \begin{picture}(0.1,1.5)
     \end{picture}}
\newsavebox{\boxBspace}
   \savebox{\boxBspace}
   { \begin{picture}(0,1.85)
     \end{picture}}

\newcommand{\paarpartbw}{\usebox{\boxpaarpartbw}}
\newcommand{\paarpartwb}{\usebox{\boxpaarpartwb}}

\newcommand{\baarpartbw}{\usebox{\boxbaarpartbw}}
\newcommand{\baarpartwb}{\usebox{\boxbaarpartwb}}

\newcommand{\singletonw}{\usebox{\boxsingletonw}}

\theoremstyle{plain} 
\newtheorem{thm}{Theorem}[section]
\newtheorem*{thm*}{Theorem}
\newtheorem*{thm*conj*}{Theorem/Conjecture}
\newtheorem{prop}[thm]{Proposition}
\newtheorem{lem}[thm]{Lemma}
\newtheorem*{lem*}{Lemma}
\newtheorem{cor}[thm]{Corollary}
\theoremstyle{definition}
\newtheorem{defn}[thm]{Definition}
\newtheorem{conjecture}[thm]{Conjecture}
\newtheorem*{defn*}{Definition}
\newtheorem{rem}[thm]{Remark}
\newtheorem{ex}[thm]{Example}
\newtheorem{question}[thm]{Question}

\newtheorem*{notation*}{Notation}

\numberwithin{equation}{section}

\renewcommand{\theta}{\vartheta}
\renewcommand{\epsilon}{\varepsilon}
\renewcommand{\subset}{\subseteq}

\newcommand{\N}{\mathbb N}

\newcommand{\C}{\mathbb C}

\newcommand{\norm}[1]{\lVert {#1}\rVert}

\newcommand{\operp}{\;\raisebox{0.04cm}{\scalebox{0.72}{$\bigcirc\textnormal{\hspace{-14pt}}\perp$}}\;}

\newcommand{\QR}[2]{
\parbox[c][33pt]{0pt}{~}{\raisebox{1ex}{\ensuremath{#1}}
\ensuremath{\mkern-5mu}\Big/\ensuremath{\mkern-5mu}
\raisebox{-1ex}{\ensuremath{#2}}}}

\DeclareMathOperator{\id}{id}

\begin{document}
\title{Models of quantum permutations}
\author{Stefan Jung and Moritz Weber }
\address{Saarland University, Fachbereich Mathematik, Postfach 151150,
66041 Saarbr\"ucken, Germany}
\email{jung@math.uni-sb.de, weber@math.uni-sb.de}
\date{\today}
\subjclass[2010]{}
\thanks{Both authors were funded by the ERC Advanced Grant NCDFP, held by Roland Speicher. The second author was also funded by the SFB-TRR 195 and the DFG project \emph{Quantenautomorphismen von Graphen}. This article is part of the first author's PhD thesis. We thank Roland Speicher and Adam Skalski for pointing out further questions related to our work; see the end of the article.}

\begin{abstract}
For \(N\ge 4\) we present a series of \(^*\)-homomorphisms \(\varphi_n:C(S_N^+)\rightarrow B_n\) where \(S_N^+\) is the quantum permutation group.
They are not necessarily representations of the quantum group \(S_N^+\) but they yield good operator algebraic models of quantum permutation matrices.
The \(C^*\)-algebras \(B_n\) allow the construction of an inverse limit \(B_{\infty}\) which defines a compact matrix quantum group \(S_N\subsetneq G\subseteq S_N^+\).
We know \(G=S_N^+\) for \(N=4,5\) from Banica's work, but we have to leave open the case \(N\ge 6\).
\end{abstract}

\maketitle

\section{Introduction}
In \cite{wang1998} Sh. Wang introduced quantum versions of the classical permutation groups \(S_N\), the so-called \emph{quantum permutation groups \(S_N^+\)}, which are examples of compact matrix quantum groups in the sense of Woronowicz, see \cite{woronowiczpseudogroups}. 
The \(C^*\)-algebra \(C(S_N^+)\) is given by the universal unital \(C^*\)-algebra
\[C(S_N^+):=C^*\big(\,u_{ij}, 1\le i,j\le N\,|\,\textnormal{\(u=(u_{ij})_{1\le i,j\le N}\in M_N\big(C(S_N^+)\big)\) is a magic unitary}\,\big)\]
where a matrix is called a \emph{magic unitary} if and only if its entries are projections summing up to \(1\) in every row and column.
The name \emph{quantum permutation group} is justified by the fact, that one obtains its classical analogue, \(S_N\) (or rather \(C(S_N)\)), by adding commutativity to the generators.
\newline
In this article we are interested in models of \(C(S_N^+)\), i.e. \(^*\)-homomorphisms \linebreak
\(\varphi:C(S_N^+)\rightarrow B\) from \(C(S_N^+)\) to \(C^*\)-algebras \(B\).
We do not require \(\varphi\) to respect the comultiplication, hence we are only interested in finding ``good'' \(^*\)-homomorphisms for the \(C^*\)-algebra \(C(S_N^+)\).
This is linked to the research on matrix models in \cite{banicanechita_flatmatrixmodels} or on Hopf images in \cite{banicabichon_hopf_images}.
As \(S_N^+\) coincides with \(S_N\) if and only if \(N\le 3\) we concentrate on the situation \(N\!\ge\! 4\).
\newline
In fact, we construct a whole series of models \((\varphi_n)_{n\in\N}\) whose kernels become smaller and smaller for increasing \(n\).
As an example for \(N\!=\!4\), we consider
\[B_n:=C^*(m^{(n)}_{ij}\,|\, 1\le i,j\le 4)\subseteq (A\otimes A)^{\otimes n}\]
where
\[A:=C^*(p,q\textnormal{ projections})\]
and
\[M_n=\left(m^{(n)}_{ij}\right)\in M_4(B_n)\]
is the magic unitary constructed by
\[M_n:=M_1^{\operp n}.\]
Here,
\begin{align*}
M_1:=&R\operp R=
\scalebox{0.9}{$\begin{pmatrix}p&0&1-p&0\\1-p&0&p&0\\0&q&0&1-q\\0&1-q&0&q\end{pmatrix}\operp \begin{pmatrix}p&0&1-p&0\\1-p&0&p&0\\0&q&0&1-q\\0&1-q&0&q\end{pmatrix}
$}
\\[8pt]
=&
\scalebox{0.9}{$
\begin{pmatrix}
\parbox{14.3cm}{
\begin{tabular}{r @{\hspace{2pt}}c@{\hspace{2pt}} l @{\hspace{20pt}} r @{\hspace{2pt}}c@{\hspace{2pt}} l @{\hspace{20pt}} r @{\hspace{2pt}}c@{\hspace{2pt}} l @{\hspace{20pt}} r @{\hspace{2pt}}c@{\hspace{2pt}} l }
$p$&$\otimes$& $p$&$(1-p)$&$\otimes$& $q$&$p$&$\otimes$& $(1-p)$&$(1-p)$&$\otimes$& $(1-q)$\\[4pt]
$(1-p)$&$\otimes$& $p$&$p$&$\otimes$&$q$&$(1-p)$&$\otimes$ &$(1-p)$&$p$&$\otimes$&$(1-q)$\\[4pt]
$q$&$\otimes$& $(1-p)$&$(1-q)$&$\otimes$& $(1-q)$&$q$&$\otimes$& $p$&$(1-q)$&$\otimes$&$q$\\[4pt]
$(1-q)$&$\otimes$& $(1-p)$&$q$&$\otimes$& $(1-q)$&$(1-q)$&$\otimes$& $p$&$q$&$\otimes$& $q$
\end{tabular}
}
\end{pmatrix}
$}
\end{align*}
and \(\operp\) is defined as in \cite{woronowiczpseudogroups}.
The choice of the matrix \(R\) might be a bit surprising for the reader familiar with the canonical proof of the noncommutativity of \(C(S_4^+)\), where one usually uses the matrix \(R\) with the second and third column swapped.
However, \(R\) behaves much better under the operation \(\operp\) (see also Example \ref{ex:the_common_model_for_C(S_4^+)}), hence our choice.
\newline
We define \(\pi_{n+1,n}:B_{n+1}\rightarrow B_n\) by dividing out the relations \(p=q=1\) in the last two legs of \(B_{n+1}\subseteq (A\otimes A)^{\otimes (n+1)}\).
The construction for general \(N\) is analogous (the definition of \(M_1\) being slightly more elaborate) and we obtain the following commuting diagram.
\begin{equation*}\label{eqn:desired_diagram}
\setlength{\unitlength}{0.5cm}
\begin{picture}(10,7)
\put(2.8,6) {$\Big(C\big(S_N^+\big),u\Big)$}
\put(-7,0) {$(B_1,M_1)\overset{\pi_{2,1}}{\xlongleftarrow{\quad\quad}}\;\cdots\;
(B_n,M_{n})\,\;\overset{\pi_{n+1,n}}{\xlongleftarrow{\quad\quad}}\; (B_{n+1},M_{n+1})\;\overset{\pi_{n+2,n+1}}{\xlongleftarrow{\quad\quad}}\;\cdots$}
\put(2,5){\vector(-2,-1){7}}
\put(6.6,5){\vector(1,-2){1.8}}
\put(3.7,5){\vector(-1,-2){1.8}}
\put(-1.4,4){$\varphi_1$}
\put(0,2.5){$\cdots$}
\put(2,4){$\varphi_n$}
\put(7.4,4){$\varphi_{n+1}\quad$}
\put(9,2.5){$\cdots$}
\end{picture}
\vspace{6pt}
\end{equation*}
%
%
%
The lower row of the diagram 
is an inverse system and admits an inverse limit \((B_{\infty},M_{\infty})\).
The central result of this article is the following.

\begin{thm*}[Theorem \ref{thm:B_infty_defines_CMQG}]  \(G:=\big(B_{\infty},M_{\infty}\big)\) is a compact matrix quantum group fulfilling \(S_N\subsetneq G\subseteq S_N^+\).
\end{thm*}

In the situation of \(N\in\{4,5\}\) it follows from the maximality of the inclusion \(S_N\subseteq S_N^+\) proved in \cite{banicauniformquantumgroups} that the constructed inverse limit is equal to \(S_N^+\).
We have to leave it open whether we have \(G=S_N^+\) for \(N\ge 6\).
\newline
Our results may be interpreted in the sense that Woronowicz's operation \(\operp\) applied iteratively to the representation \(R\) as above is powerful enough to finally reproduce all of \(C(S_N^+)\), at least in the cases \(N=4\) and \(N=5\).
Hence, \(R^{\operp n}\) yields good models of quantum permutation matrices for practical purposes such as in \cite{lupini_etall}, see also Section \ref{subsec:more_models_of_C(S_4^+)}.
\newline
In Section 5, we comment on how to generalize the presented ideas and results in the situation of easy quantum groups, of which the (quantum) permutation groups are special cases.
We show that the construction of an inverse system and the corresponding inverse limit as above can be performed whenever a suitable starting pair \((B_1,M_1)\) is given.
\section{Preliminaries}
In this section we define compact matrix quantum groups and quantum permutation groups \(S_N^+\).
Throughout this work, let \(\otimes\) denote the minimal tensor product of \(C^*\)-algebras and we write \([n]\) for the set of natural numbers \(\{1,\ldots,n\}\).
\subsection{The categories \(\mathcal{C}_N\) and compact matrix quantum groups}
\begin{defn}\label{defn:category_C_N}
Consider for given \(N\in\N\) the category \(\mathcal{C}_N\) whose objects are pairs \((D,M)\) where 
\begin{itemize}
\item 
\(D\) is a \(C^*\)-algebra,
\item
\(M\) is an \(N\!\times\!N\)-matrix over \(D\), i.e. \(M\in M_N(D)=M_N(\C)\otimes D\) and
\item the \(C^*\)-algebra \(D\) is generated by the \(N^2\) entries of \(M\).
\end{itemize}
Arrows in \(\mathcal{C}_N\) between objects \((D_1,M_1)\) and \((D_2,M_2)\) are \(^*\)-homomorphisms \(\varphi\) sending the entries of \(M_1\) canonically onto the entries of \(M_2\), i.e.
\[(\id\otimes \varphi)M_1=M_2.\]
\end{defn}
In \cite{woronowiczpseudogroups} Woronowicz defined \(C^*\)-algebraic compact matrix quantum groups.
\begin{defn}\label{defn:CMQGs}
Let \(N\in \N\) and let \((A,u)\) be an object in the category \(\mathcal{C}_N\) such that \(u\) is a unitary and \(u^{(*)}:=\big(u_{ij}^*\big) \) is invertible.
Assume that there exists a unital \(^*\)-homomorphism \(\Delta:A\rightarrow A\otimes A\) (called comultiplication on \(A\)) that fulfils
\begin{equation}\label{eqn:structure_of_comultiplication}
\Delta(u_{ij})=\sum_{k=1}^{N}u_{ik}\otimes u_{kj}\quad\quad \forall 1\le i,j\le N.
\end{equation}
Then we denote \(A\) also by \(C(G)\), \(u\) by \(u_G\) and \(G\!:=\!\big(C(G),u_G\big)\) is a \emph{compact matrix quantum group of size \(N\)}.
\end{defn}
Compact matrix quantum groups are generalizations of (unitary) compact matrix groups, compare \cite[Prop. 6.1.10]{timmermann}.
\subsection{Quantum permutation groups}
We now come to the definition of the objects of interest in this work, the quantum permutation groups as defined by Wang in \cite{wang1998}.
\begin{defn}\label{defn:S_N^+}
\begin{itemize}
\item[(a)]
Given for some \(N\in\N\) a matrix \(u\!=\!(u_{ij})_{1\le i,j\le N}\) with entries in some unital \(^*\)-algebra, we call \(u\) a \emph{magic unitary} if its entries are projections (i.e. \(u_{ij}=u_{ij}^*=u_{ij}^2\)) that sum up to \(1\) in every row and column.
\item[(b)]
Let \(N\in\N\) and \(u\) be an \(N\!\times\!N\)-matrix of generators.
Define
\[A:=C^*\big(\,u_{ij},1\le i,j\le N\,|\, u\textnormal{ is a magic unitary}\,\big).\]
Then we call the compact matrix quantum group \(S_N^+:=(A,u)\) the \emph{quantum permutation group} and write \(A=C(S_N^+)\).
\end{itemize}
\end{defn}
\begin{rem}\label{rem:S_N^+_becomes_S_N_when_adding_commutativity}
When we divide out in \(A\!=\!C(S_N^+)\) the commutativity relations, we obtain \(S_N\), seen as the compact matrix quantum group \(S_N=\big(C(S_N),u_{S_N}\big)\).
\end{rem}
\begin{defn}\label{defn:subgrouprelation}
Given an object \(G=(A,u)\) in \(\mathcal{C}_N\) that is a compact matrix quantum group, we say that it fulfils the \emph{subgroup relation} \(S_N\subseteq G\subseteq S_N^+\) if and only if there exists a diagram
\[\big(C(S_N^+),u_{S_N^+}\big)\longrightarrow\big(A,u\big)\longrightarrow\big(C(S_N),u_{S_N}\big).\]
We write in addition \(S_N\subsetneq G\) if the arrow from \(\big(A,u\big)\) to \(\big(C(S_N),u_{S_N}\big)\) is not injective (i.e. not invertible).
\end{defn}
For \(N\le 3\) the compact matrix quantum groups \(S_N^+\) and \(S_N\) coincide.
For \(N\!\ge\!4\) however, we have \(S_N\subsetneq S_N^+\) in the sense that there is a non-injective \(^*\)-homomorphism \(C(S_N^+)\rightarrow C(S_N)\) sending generators to generators.
Recall that \(S_N^+\) is not a group in this case, i.e. \(C(S_N^+)\) is a noncommutative \(C^*\)-algebra.
Indeed, \(\phi:C(S_4^+)\rightarrow C^*(p,q\textnormal{ projections})\) mapping \(u\) to
\[R:=\begin{pmatrix}p&0&1-p&0\\1-p&0&p&0\\0&q&0&1-q\\0&1-q&0&q\end{pmatrix}\]
is a surjective \(^*\)-homomorphism onto a noncommutative \(C^*\)-algebra.
Usually, one uses a variant of \(R\) where the second and third columns are swapped, but we prefer this matrix \(R\) for later purposes.

\pagebreak

\section{A sequence of models for \(C(S_N^+)\)}
\label{sec:a_sequence_of_models}
For the rest of this article let \(N\in\N_{\ge4}\).
\subsection{The \(\;\raisebox{0.04cm}{$\bigcirc\textnormal{\hspace{-14pt}}\perp$}\;\)-product}
\label{subsec:the_operp-product}
We start by defining the so-called \(\operp\)-product of matrices, compare \cite{woronowiczpseudogroups}.
\begin{defn}
Given two matrices \(M_1\in M_N(A)\) and \(M_2\in M_N(B)\) for two \(C^*\)-algebras \(A\) and \(B\), we define the matrix
\[M_1\operp M_2:=\sum_{i,j=1}^{N}E_{ij}\otimes\left(\sum_{k=1}^{N} m^{(1)}_{ik}\otimes  m^{(2)}_{kj}\right),\]
an \(N\!\times\! N\)-matrix with entries in \(A\otimes B\) (or any suitable \(C^*\)-subalgebra of it).
\end{defn}
\begin{lem}\label{lem:operp-product_gives_again_magic_unitary}
If \(M_1\) and \(M_2\) are magic unitaries, so is \(M_1\operp M_2\).
\end{lem}
\begin{proof}
Straightforward.
\end{proof}
\subsection{The matrix \(M_1\)}
\label{subsec:the_matrix_M_1}
\begin{defn}\label{defn:R_(a,b),(c,d)}
\begin{itemize}
\item[(a)]
We define the \(C^*\)-algebra \(A\) by
\begin{equation}\label{eqn:defn_of_A}
A:=C^*\big( p,q,1\,|\, p,q,1\textnormal{ projections},\;1p=p1=p\,,\,q1=1q=q \big),
\end{equation}
the universal unital \(C^*\)-algebra generated by two projections.
\item[(b)] Consider \(N\!\ge 4\) and let \(1\le a,b,c,d\le N\) be pairwise different.
We define the matrix
\[R_{(a,b),(c,d)}\in M_N(A)\]
by the following properties:
\begin{itemize}
\item[(1)] The entries \((a,a)\) and \((b,b)\) are given by \(p\).
\item[(2)] The entries \((c,c)\) and \((d,d)\) are given by \(q\).
\item[(3)] All other diagonal entries are equal to \(1\).
\item[(4)] The entries \((a,b)\) and \((b,a)\) are given by \(1\!-\!p\).
\item[(5)] The entries \((c,d)\) and \((d,c)\) are given by \(1\!-\!q\).
\item[(6)] All other off-diagonal entries are zero.
\end{itemize}
\end{itemize}
\end{defn}
\begin{ex}
In the case \(N\!=\!5\) only one entry in a matrix \(R_{(a,b),(c,d)}\) is equal to \(1\).
For example we have
\[
R_{(1,4),(3,5)}=
\begin{pmatrix}
p&0&0&1\!-\!p&0\\
0&1&0&0&0\\
0&0&q&0&1\!-\!q\\
1\!-\!p&0&0&p&0\\
0&0&1\!-\!q&0&q
\end{pmatrix}\in M_5(A).
\]
\end{ex}
\begin{defn}\label{defn:(B_1,M_1)}
Let \(L\) be a natural number such that each permutation \(\sigma\in S_N\) can be written as a product of at most \(L\) transpositions.
For \(1\le a,b\le  N\) let \(R_{(a,b),*}\) be a matrix \(R_{(a,b),(c,d)}\) where \(1\le c,d\le N\) are chosen such that \(a,b,c,d\) are pairwise different.
We define the object \((B_1,M_1)\) in the category \(\mathcal{C}_N\) by
\[M_1=\left(m^{(1)}_{ij}\right)_{1\le i,j\le N}:=\left(\underset{\scalebox{0.68}{\(\begin{matrix}\\[-10pt]1\!\le\! a\!<\!b\!\le\!N\end{matrix}\)}}{\scalebox{2}{$\operp$}} R_{(a,b),*}\right)^{\operp L}\]
and
\[B_1:=C^*\big(\,m^{(1)}_{ij}\;|\;1\le i,j\le N\big)\subseteq A^{\otimes\frac{LN(N-1)}{2}}.\]
\end{defn}
\begin{lem}\label{lem:proving_all_properties_of_(B_1,M_1)}
Consider the object \((B_1,M_1)\) in \(\mathcal{C}_N\) from Definition \ref{defn:(B_1,M_1)}.
There exists a diagram of the form
\[(C(S_N^+),u_{S_N^+})\overset{\varphi_1}\longrightarrow (B_1,M_1)\overset{\phi}\longrightarrow (C(S_N),u_{S_N}).\]
\end{lem}
\begin{proof}
The existence of \(\varphi_1\) is by Lemma \ref{lem:operp-product_gives_again_magic_unitary}.
In order to prove existence of the arrow \((B_1,M_1)\overset{\phi}\longrightarrow (C(S_N),u_{S_N})\), we start with the matrices \(R_{(a,b),*}\), the matrix \(M_1\) and the \(C^*\)-algebra \(B_1\).
Dividing out in all appearing legs the commutativity relations \(pq=qp\), we obtain matrices \(R'_{(a,b),*}\), a matrix \(M'_1\) and a commutative \(C^*\)-algebra \(B'_1\).
The corresponding quotient map \(\phi':B_1\rightarrow B'_1\) is an arrow 
\[(B_1,M_1)\overset{\phi'}{\longrightarrow}(B'_1,M'_1)\]
and \(M'_1\) is magic.
By Remark \ref{rem:S_N^+_becomes_S_N_when_adding_commutativity}, we have an arrow 
\[(C(S_N),u_{S_N})\overset{\rho}{\longrightarrow}(B'_1,M'_1)\]
because \(M'_1\) is magic and its entries pairwisely commute.
It remains to prove the following claim:
\[(*)\quad\forall \sigma\in S_N:\quad \phi'(m_{\sigma})\neq 0\quad\textnormal{ where }m_{\sigma}:=m^{(1)}_{1\sigma(1)}m^{(1)}_{2\sigma(2)}\cdot\ldots\cdot m^{(1)}_{n\sigma(n)}\in B_1.\]
This shows that the vector space dimension of \(\phi'(B_1)\) is \(N!\!=\!\dim\big(C(S_N)\big)\) and the arrow \(\rho\) above is invertible.
The composition \(\phi:=\rho^{-1}\circ \phi'\) is the desired arrow from \((B_1,M_1)\) to \(\big(C(S_N),u_{S_N}\big)\).
\newline
In order to prove the statement \((*)\) for given \(\sigma\in S_N\), we define
\[m'_{\sigma}:=\phi'(m_\sigma)\]
and observe that it suffices to construct a \(^*\)-homomorphism
\[\mu:B'_1\longrightarrow \C\]
that fulfils \(\mu(m'_ {\sigma})=1\).
We write \(\sigma^{-1}\) as a product of \(l\) transpositions \(\tau_{\alpha,\beta}=(\alpha,\beta)\in S_N\) with \(\alpha\!<\!\beta\) such that \(l\) is as small as possible:
\begin{equation}\label{eqn:sigma_as_product_of_transpositions}
\sigma^{-1}=\tau_{\alpha_1,\beta_1}\tau_{\alpha_2,\beta_2}\cdots\tau_{\alpha_l,\beta_l}=(\alpha_1,\beta_1)(\alpha_2,\beta_2)\cdots(\alpha_l,\beta_l)
\end{equation}
By definition of \(M_1\) and \(M'_1\), it holds
\[M'_1=\left(\underset{\scalebox{0.68}{\(\begin{matrix}\\[-10pt]1\!\le\! a\!<\!b\!\le\!N\end{matrix}\)}}{\scalebox{2}{$\operp$}} R'_{(a,b),*}\right)^{\operp L}\]
with \(L\ge l\).
Writing out this \(\operp\)-product, we see that \(M'_1\) is of the form
\begin{equation}\label{eqn:writing_out_M'_1}
M'_1:=\ldots\operp R'_{(\alpha_1,\beta_1),*}\operp\ldots\operp R'_{(\alpha_2,\beta_2),*}\operp\ldots\quad\ldots\operp R'_{(\alpha_l,\beta_l),*}\operp\ldots,
\end{equation}
i.e. we find, among other \(\operp\)-factors, matrices \(R'_{(\alpha_1,\beta_1),*}\), \(R'_{(\alpha_2,\beta_2),*}\), \(\ldots\),\(R'_{(\alpha_l,\beta_l),*}\) that appear  from left to right in this order.
Let's say these matrices appear in the \(\operp\)-product from Equation \ref{eqn:writing_out_M'_1} at positions \(k_1,\ldots,k_l\).
Define a quotient map \(\mu\) on \(B'_1\) in the following way:
\begin{itemize}
\item[(i)] In each of the legs \(k_1,\ldots,k_l\) we apply a quotient map
\(\mu_1\) that divides out exactly the relation \(1\!-\!p\!=\!q\!=\!\mathds{1}\). Note that we have
\[(1\otimes \mu_1)(R'_{(a,b),*})=\tau_{a,b}.\]
\item[(ii)] In each of the remaining legs we apply a quotient map \(\mu_0\) that divides out exactly the relations \(p\!=\!q\!=\!1\).
Note that we have
\[(1\otimes \mu_0)(R'_{(a,b),*})=\mathds{1}_{S_N}.\]
\end{itemize}
From these observations we directly deduce 
\[(1\otimes\mu)M'_1=\tau_{\alpha_1,\beta_1}\tau_{\alpha_2,\beta_2}\cdots\tau_{\alpha_l,\beta_l}=\sigma^{-1}.\]
Recall that a permutation matrix \(\sigma\) fulfils
\[\sigma_{ij}=\delta_{i,\sigma(j)},\]
so it holds
\[(\sigma^{-1})_{i\sigma(i)}=\delta_{i,(\sigma^{-1}\circ\sigma)(i)}=1.\]
We conclude
\[\mu\big(m'_\sigma\big)=(\sigma^{-1})_{1\sigma(1)}\cdots (\sigma^{-1})_{N\sigma(N)}=1,\]
and thus \(m'_\sigma=\phi'(m_\sigma)\) is non-zero, proving \((*)\).
\end{proof}
\begin{rem}\label{rem:existence_of_nu}
Obviously, there is an arrow 
\[\big(C(S_N),u_{S_N}\big)\overset{\nu'}{\longrightarrow}\big(\C,\mathds{1}_{M_N(\C)}\big)\]
because \(\mathds{1}_{M_N(\C)}\) is a magic unitary with commuting entries.
The composition \(\nu:=\nu'\circ\phi\) with \(\phi\) from Lemma \ref{lem:proving_all_properties_of_(B_1,M_1)} is an arrow
\[(B_1,M_1)\overset{\nu}\longrightarrow (\C,\mathds{1}_{M_N(\C)}).\]
\end{rem}
\subsection{The matrices \(M_n\)}
\label{subsec:the_matrices_M_n}
\begin{defn}\label{defn:M_n_and_B_n}
Consider the object \((B_1,M_1)\) in \(\mathcal{C}_N\) from Definition \ref{defn:(B_1,M_1)}.
Define 
\[M_n=\left(m^{(n)}_{ij}\right)_{1\le i,j\le N}:=M_1^{\operp n}\]
and
\[B_n:=C^*(\,m^{(n)}_{ij}\,|\,1\le i,j\le N)\subseteq B_1^{\otimes n}.\]
\end{defn}
By Remark  \ref{rem:existence_of_nu}, we have for every \(n\in\N\) an arrow
\[(B_{n+1},M_{n+1})\overset{\pi_{n+1,n}}{\longrightarrow}(B_n,M_n)\]
given by restricting \((\id_{B_1})^{\otimes n}\otimes \nu\) to \(B_{n+1}\).
We obtain a commuting diagram of the form
\begin{equation*}\label{eqn*:desired_diagram}
\setlength{\unitlength}{0.5cm}
\begin{picture}(10,7)
\put(2.8,6) {$\Big(C\big(S_N^+\big),u_{S_N^+}\Big)$}
\put(-7,0) {$(B_1,M_1)\overset{\pi_{2,1}}{\xlongleftarrow{\quad\quad}}\;\cdots\;
(B_n,M_{n})\,\;\overset{\pi_{n+1,n}}{\xlongleftarrow{\quad\quad}}\; (B_{n+1},M_{n+1})\;\overset{\pi_{n+2,n+1}}{\xlongleftarrow{\quad\quad}}\;\cdots$}
\put(2,5){\vector(-2,-1){7}}
\put(6.6,5){\vector(1,-2){1.8}}
\put(3.7,5){\vector(-1,-2){1.8}}
\put(-1.4,4){$\varphi_1$}
\put(0,2.5){$\cdots$}
\put(2,4){$\varphi_n$}
\put(7.4,4){$\varphi_{n+1}\quad$}
\put(9,2.5){$\cdots$}
\end{picture}
\vspace{6pt}
\end{equation*}
In particular, by Lemma \ref{lem:operp-product_gives_again_magic_unitary}, every pair \((B_n,M_n)\) defines a model of \(C(S_N^+)\).
\begin{rem}\label{rem:existence_of_arrow_phi_n}
Considering \(\phi_1:B_1\rightarrow C(S_N)\) as described in Lemma \ref{lem:proving_all_properties_of_(B_1,M_1)}, the composition
\[\phi\circ\pi_{2,1}\circ\ldots\circ\pi_{n,n-1}\]
is an arrow from \((B_n,M_n)\) to \((C(S_N),u_{S_N})\).
\end{rem}
\subsection{More models of \(C(S_4^+)\)}
\label{subsec:more_models_of_C(S_4^+)}
We end this section by listing further models of \(C(S_4^+)\).
They may be used in order to obtain additional models in the general case of \(C(S_N^+)\) by filling up the diagonal with units.
In the following let \(A\) always be the \(C^*\)-algebra as in Definition \ref{defn:R_(a,b),(c,d)}.
\begin{ex}\label{ex:the_common_model_for_C(S_4^+)}
The idea of the matrices \(R_{(a,b),(c,d)}\) (and the associated models \((A,R_{(a,b),(c,d)})\)) originates in the matrix
\[
\widehat{R}:=\begin{pmatrix}p&1-p&0&0\\1-p&p&0&0\\0&0&q&1-q\\0&0&1-q&q\end{pmatrix}
\]
which is usually taken into account when proving non-commutativity of \(C(S_4^+)\).
However, we have
\[
\widehat{R}\operp \widehat{R}=
\scalebox{0.8}{$
\begin{pmatrix}
\parbox{15.6cm}{
\begin{tabular}{r @{\hspace{2pt}}c@{\hspace{2pt}} l @{\hspace{20pt}} r @{\hspace{2pt}}c@{\hspace{2pt}} l @{\hspace{20pt}} r @{\hspace{2pt}}c@{\hspace{2pt}} l @{\hspace{20pt}} r @{\hspace{2pt}}c@{\hspace{2pt}} l }
$p$&$\otimes$& $p$&$p$&$\otimes$& $(1-p)$&&\begin{picture}(0,0)\put(-0.2,-0.7){0}\end{picture}&&&\begin{picture}(0,0)\put(-0.2,-0.7){0}\end{picture}&\\[4pt]
$+(1-p)$&$\otimes$& $(1-p)$&$+(1-p)$&$\otimes$&$p$&&&&&&\\[10pt]

$(1-p)$&$\otimes$& $p$&$(1-p)$&$\otimes$& $(1-p)$&&\begin{picture}(0,0)\put(-0.2,-0.7){0}\end{picture}&&&\begin{picture}(0,0)\put(-0.2,-0.7){0}\end{picture}&\\[4pt]
$+p$&$\otimes$& $(1-p)$&$+p$&$\otimes$&$p$&&&&&&\\[10pt]

&\begin{picture}(0,0)\put(-0.2,-0.7){0}\end{picture}&&&\begin{picture}(0,0)\put(-0.2,-0.7){0}\end{picture}&&$q$&$\otimes$& $q$&$q$&$\otimes$&$(1-q)$\\[4pt]
&&&&&&$+(1-q)$&$\otimes$& $(1-q)$&$+(1-q)$&$\otimes$&$q$\\[10pt]
&\begin{picture}(0,0)\put(-0.2,-0.7){0}\end{picture}&&&\begin{picture}(0,0)\put(-0.2,-0.7){0}\end{picture}&&$(1-q)$&$\otimes$& $q$&$(1-q)$&$\otimes$&$(1-q)$\\[4pt]
&&&&&&$+q$&$\otimes$& $(1-q)$&$+q$&$\otimes$&$q$\\
\end{tabular}
}
\end{pmatrix}
$}
\]
and the corresponding object in \(\mathcal{C}_4\) is equivalent to \((A,\widehat{R})\).
Therefore, any \(\operp\)-product of matrices \(\widehat{R}\) gives an object \((\widehat{B}_n,\widehat{M}_n)\) equivalent to \((A,\widehat{R})\), i.e. the sequence of models \(\varphi_n:C(S_4^+)\rightarrow \widehat{B}_n\) exists, but, as a sequence, it is trivial.
Note that the inverse of \(\widehat{\pi}_{2n,n}\), the arrow
\[(\widehat{B}_n,\widehat{M}_n)\overset{\widehat{\pi}_{2n,n}^{-1}}{\xrightarrow{\quad\quad\quad\quad}}(\widehat{B}_{2n},\widehat{M}_n\operp \widehat{M}_n)\]
defines a comultiplication on \(\widehat{B}_n\) such that \((\widehat{B}_n,\widehat{M}_n)\) becomes a compact matrix quantum group. This seems \emph{not} to be the case for the objects  \((B_m,M_m)\) from Definition \ref{defn:M_n_and_B_n}, see also Remark \ref{RemModels} below.
\end{ex}
\begin{ex}\label{ex:R_from-introduction}
In order to obtain a model for \(C(S_4^+)\),  the symbols \(p\) and \(q\) do not have to be on the diagonal of \(R_{(a,b),(c,d)}\).
The matrix
\[R:=\begin{pmatrix}p&0&1-p&0\\1-p&0&p&0\\0&q&0&1-q\\0&1-q&0&q\end{pmatrix}\]
from the introduction gives an example for such a matrix.
\end{ex}
\begin{ex}\label{ex:second_power_of_R}
As mentioned in the introduction, the second \(\operp\)-power of \(R\) is
\[R\operp R=\scalebox{0.9}{$
\begin{pmatrix}
\parbox{14.3cm}{
\begin{tabular}{r @{\hspace{2pt}}c@{\hspace{2pt}} l @{\hspace{20pt}} r @{\hspace{2pt}}c@{\hspace{2pt}} l @{\hspace{20pt}} r @{\hspace{2pt}}c@{\hspace{2pt}} l @{\hspace{20pt}} r @{\hspace{2pt}}c@{\hspace{2pt}} l }
$p$&$\otimes$& $p$&$(1-p)$&$\otimes$& $q$&$p$&$\otimes$& $(1-p)$&$(1-p)$&$\otimes$& $(1-q)$\\[4pt]
$(1-p)$&$\otimes$& $p$&$p$&$\otimes$&$q$&$(1-p)$&$\otimes$ &$(1-p)$&$p$&$\otimes$&$(1-q)$\\[4pt]
$q$&$\otimes$& $(1-p)$&$(1-q)$&$\otimes$& $(1-q)$&$q$&$\otimes$& $p$&$(1-q)$&$\otimes$&$q$\\[4pt]
$(1-q)$&$\otimes$& $(1-p)$&$q$&$\otimes$& $(1-q)$&$(1-q)$&$\otimes$& $p$&$q$&$\otimes$& $q$
\end{tabular}
}
\end{pmatrix}
$}\]
and it is obviously not equivalent to the model given by \(R\).
\end{ex}
\begin{ex}\label{ex:third_power_of_R}
The third \(\operp\)-power of \(R\) is given by
\[
\scalebox{0.68}{
\(
\begin{pmatrix}
p\otimes p\otimes p & p\otimes (1-p)\otimes q & p\otimes p\otimes (1-p)  & p\otimes (1-p)\otimes (1-q)\\
+(1-p)\otimes q\otimes(1-p) & +(1-p)\otimes (1-q)\otimes (1-q) & +(1-p)\otimes q\otimes p & +(1-p)\otimes (1-q)\otimes q\\
&&&\\
&&&\\
(1-p)\otimes p\otimes p & (1-p)\otimes (1-p)\otimes q & (1-p)\otimes p\otimes (1-p) & (1-p)\otimes (1-p)\otimes (1-q)\\
+p\otimes q\otimes(1-p) & +p\otimes (1-q)\otimes (1-q) & +p\otimes q\otimes p & +p\otimes (1-q)\otimes q\\
&&&\\
&&&\\
q\otimes (1-p)\otimes p & q\otimes p\otimes q & q\otimes (1-p)\otimes (1-p) & q\otimes p\otimes (1-q)\\
+(1-q)\otimes (1-q)\otimes(1-p) & +(1-q)\otimes q\otimes (1-q) & +(1-q)\otimes (1-q)\otimes p & +(1-q)\otimes q\otimes q\\
&&&\\
&&&\\
(1-q)\otimes (1-p)\otimes p & (1-q)\otimes p\otimes q & (1-q)\otimes (1-p)\otimes (1-p) & (1-q)\otimes p\otimes (1-q)\\
+q\otimes (1-q)\otimes(1-p) & +q\otimes q\otimes (1-q) & +q\otimes (1-q)\otimes p & +q\otimes q\otimes q\\
\end{pmatrix}
\)
}
\]
If \(A_3\subset A^{\otimes 3}\) denotes the \(C^*\)-subalgebra generated by the entries of \(R^{\operp 3}\), then the matrix \(R^{\operp 3}\) allows an arrow 
\[\big(A_3,R^{\operp 3}\big)\overset{\phi_3}{\longrightarrow}\big(C(S_4),u_{S_4}\big).\]
The proof is analogous to the proof of part (c) in Lemma \ref{lem:proving_all_properties_of_(B_1,M_1)} apart from the fact that the claim \((*)\) can be directly checked in the present case.
Consequently, all \(R^{\operp n}\) with \(n\ge 3\) allow a corresponding arrow \(\phi_n\) to \(\big(C(S_4),u_{S_4}\big)\).
\end{ex}

\begin{rem}\label{RemModels}
Let us comment a bit on the matrix \(R\) and its powers as in Examples \ref{ex:R_from-introduction}, \ref{ex:second_power_of_R} and \ref{ex:third_power_of_R}.
\begin{itemize}
\item[(a)]
The pair \((A,R)\) does not allow an arrow \(\nu\) to \((\C,\mathds{1}_{M_N(\C)})\), hence we cannot define the arrows \(\pi_{n+1,n}\) as in Definition \ref{defn:M_n_and_B_n} and their existence is unclear.
This is why we focused on even \(\operp\)-powers of \(R\) and defined in the introduction \(M_1\!:=\!R\operp R\).
Roughly speaking, every \(\operp\)-multiplication with \(R\) in some sense ``swaps'' the second and third column such that every second \(\operp\)-power of \(R\) has the \(p\)'s and \(q\)'s in the right places.
\item[(b)]
Neither the pair \((A,R)\) nor \((B_1,M_1)=(B_1,R\operp R)\) allows an arrow \(\phi\) to \(\big(C(S_N),u_{S_N}\big)\).
This is clear for \(R\) as it has vanishing entries.
For \(M_1=R\operp R\) this follows for example from \(m^{(1)}_{11}m^{(1)}_{23}=0\), see Example \ref{ex:second_power_of_R}.
\item[(c)]
In virtue of Section \ref{sec:the_limit_object} following hereafter, the sequence \((B_n,M_n)\) could be used alternatively to the one from Definition \ref{defn:M_n_and_B_n}. Indeed, in order to obtain Theorem \ref{thm:B_infty_defines_CMQG}, it suffices that there is a pair \((B_m,M_m)\) allowing for a (non-injective) arrow \(\phi_m\) to \((C(S_4),u_{S_4})\). 
\end{itemize}
\end{rem}
\section{The limit object \((B_{\infty},M_{\infty})\)}\label{sec:the_limit_object}
\subsection{Inverse limits of inverse systems}
\label{subsec:inverse_limits_of_inverse_systems}
Inverse systems and inverse limits can be defined in a much more general context, see for example \cite{philipps_inverselimit} and the references mentioned there.
However, we stick to a very special situation such that its description and the proof of existence becomes easy to handle.
\vspace{11pt}
\newline
Consider for \(N\in\N\) the category \(\mathcal{C}_N\) from Definition \ref{defn:category_C_N}.
We call a diagram of the form
\begin{equation*}\label{eqn:inverse_system}
(D_1,M_1)\overset{\pi_{2,1}}{\xlongleftarrow{\quad\quad}}\cdots \overset{\pi_{n,n-1}}{\xlongleftarrow{\quad\quad}}(D_n,M_n) \overset{\pi_{n+1,n}}{\xlongleftarrow{\quad\quad}}(D_{n+1},M_{n+1})\overset{\pi_{n+2,n+1}}{\xlongleftarrow{\quad\quad}}\cdots\quad
\end{equation*}
an \emph{inverse system}.
Recall, see for example \cite{maclane_categories}, that the limit of a diagram as above is the minimal object \((D_{\infty},M_{\infty})\) in \(\mathcal{C}_N\) that allows a commuting diagram of the form
\begin{equation}\label{eqn:diagram_for_(D_infty,M_infty)}
\setlength{\unitlength}{0.5cm}
\raisebox{-1.5cm}{\begin{picture}(10,7)
\put(3.2,6) {$(D_{\infty},M_{\infty})$}
\put(-7,0) {$(D_1,M_1)\overset{\pi_{2,1}}{\xlongleftarrow{\quad\quad}}\;\cdots\;
(D_n,M_{n})\,\;\overset{\pi_{n+1,n}}{\xlongleftarrow{\quad\quad}}\; (D_{n+1},M_{n+1})\;\overset{\pi_{n+2,n+1}}{\xlongleftarrow{\quad\quad}}\;\cdots$}
\put(2.8,5.4){\vector(-2,-1){7.2}}
\put(6.4,5.4){\vector(1,-2){2}}
\put(3.9,5.4){\vector(-1,-2){2}}
\put(-1.4,4){$\phi_1$}
\put(0,2.5){$\cdots$}
\put(2,4){$\phi_n$}
\put(7.4,4){$\phi_{n+1}\quad$}
\put(9,2.5){$\cdots$}
\end{picture}.
}\vspace{6pt}
\end{equation}
Minimality says that for every other object \((B,M)\) that allows a diagram of this form,
\begin{equation}\label{eqn:diagram_for_(B,M)}
\setlength{\unitlength}{0.5cm}
\raisebox{-1.5cm}{\begin{picture}(10,7)
\put(3.6,6) {$(B,M)$}
\put(-7,0) {$(D_1,M_1)\overset{\pi_{2,1}}{\xlongleftarrow{\quad\quad}}\;\cdots\;
(D_n,M_{n})\,\;\overset{\pi_{n+1,n}}{\xlongleftarrow{\quad\quad}}\; (D_{n+1},M_{n+1})\;\overset{\pi_{n+2,n+1}}{\xlongleftarrow{\quad\quad}}\;\cdots$}
\put(3,5.5){\vector(-2,-1){7}}
\put(6.4,5.5){\vector(1,-2){1.8}}
\put(3.9,5.5){\vector(-1,-2){1.8}}
\put(-1.4,4){$\psi_1$}
\put(0,2.5){$\cdots$}
\put(2,4){$\psi_n$}
\put(7.4,4){$\psi_{n+1}\quad$}
\put(9,2.5){$\cdots$}
\end{picture},
}\vspace{6pt}
\end{equation}
there exists an arrow 
\[(B,M)\overset{\psi}{\longrightarrow}(D_{\infty},M_{\infty}),\]
such that each arrow \(\phi_n\) in Diagram \ref{eqn:diagram_for_(D_infty,M_infty)} factors through \(\psi\), i.e.  for every \(n\in\N\) the following diagram commutes:
\begin{center}
\begin{minipage}{150mm}
~\vspace{1cm}

\begin{equation}\label{eqn:universal_property_of_(D_infty,M_infty)}
\setlength{\unitlength}{0.5cm}
\raisebox{-3cm}{\begin{picture}(10,7)
\put(3.4,9.5) {$\big(\,B\,,\,M\,\big)$}
\put(3,5.8) {$\big(\,D_{\infty}\,,\,M_{\infty}\,\big)$}
\put(-0.8,2) {$(\,D_n\,,\,M_{n}\,)\,\;\overset{\pi_{n+1,n}}{\xlongleftarrow{\quad\quad\quad\;}}\; (D_{n+1},M_{n+1})$}
\put(5,9){\vector(0,-1){2}}
\put(3.3,9){\vector(-1,-2){3}}
\put(6.7,9){\vector(1,-2){3}}
\put(3.7,5){\vector(-1,-2){1}}
\put(6.6,5){\vector(1,-2){1}}
\put(5.2,7.8){$\scalebox{0.8}{$\psi$}$}
\put(1,6.5){$\scalebox{0.8}{$\psi_n$}$}
\put(8.6,6.5){$\scalebox{0.8}{$\psi_{n+1}$}$}
\put(2.2,4){$\scalebox{0.8}{$\phi_n$}$}
\put(7.2,4){$\scalebox{0.8}{$\phi_{n+1}$}\quad$}
\end{picture}
}
\end{equation}
\end{minipage}
\end{center}
\begin{lem}\label{lem:existence_of_inverse_limits}
Let \(\Big(\,\big((D_n,M_n)\big)_{n\in\N}\,,\,\big(\pi_{n+1,n}\big)_{n\in\N}\,\Big)\) be an inverse system in \(\mathcal{C}_N\).
Denote for \(n\!\in\!\N\) the entries of \(M_n\) with \(m^{(n)}_{ij}\).
If for all \(1\!\le\!i,j\!\le\!N\) the sequence of \((i,j)\)-th entries \((m^{(n)}_{ij}\big)_{n\in\N}\) is bounded, then the limit \((D_{\infty},M_{\infty})\) of the inverse system exists.
We denote it by 
\[\lim_{\infty\leftarrow n}(D_n,M_n):=(D_{\infty},M_{\infty})\]
and call it the inverse limit of the given inverse system.
\end{lem}
\begin{proof}
Existence and uniqueness is not difficult to prove, see for example \cite{philipps_inverselimit}.
However, to keep this work self-contained, we present an own proof.
\newline
We start with the proof of existence.
\newline
\textbf{Step 1: Construction of \((D_{\infty},M_{\infty})\):}
Consider the free \(^*\)-algebra \(\mathcal{D}\) generated by \(N^2\) symbols \(m_{ij}^{(\infty)}\) with \(1\le i,j\le N\) and let \(\phi'_n:\mathcal{D}\rightarrow D_n\) be the \(^*\)-homomorphism  given by the mapping \(\phi'_n(m_{ij}^{(\infty)})=m_{ij}^{(n)}\).
Impose on \(\mathcal{D}\) the \(C^*\)-seminorms
\[f_n:=\norm{\phi'_n(\cdot)}_{D_n}\]
as well as
\[f:=\sup_{n\in\N}{f_n}.\]
Note that \((f_n)_{n\in\N}\) is bounded pointwise by assumption, so \(f\) exists. 
We have \(\phi'_n=\pi_{n+1,n}\circ\phi'_{n+1}\) and  \(\pi_{n+1,n}\) is norm-decreasing as it is a \(^*\)-homomorphism.
Therefore, the sequence \(\big(f_n\big)_{n\in\N}\) is increasing and the supremum that defines \(f\) is in fact a limit.
Evidently, \(f\) gives a \(C^*\)-norm on the quotient \(\mathcal{D}_{\infty}:=\mathcal{D}/\ker(f)\) and we define \(D_{\infty}\) to be its completion.
Considering the \(m^{(\infty)}_{ij}\) as elements in \(D_{\infty}\) and defining \(M_{\infty}:=\big(m^{(\infty)}_{ij}\big)_{1\le i,j\le N}\), the pair \((D_{\infty},M_{\infty})\) is an object in our category \(\mathcal{C}_N\).
\newline
Existence of the arrows 
\[(D_{\infty},M_{\infty})\overset{\phi_n}{\longrightarrow}(D_n,M_n)\]
for every \(n\in\N\) can now be proved as follows:
Firstly, we have \(\ker(f)\subseteq \ker(f_n)\) because \(f:=\sup f_n\).
Secondly, it holds \(\ker(f_n)=\ker(\phi'_n)\) because \(f_n:=\norm{\phi'_n(\cdot)}_{D_n}\).
\newline
We conclude
\[\QR{\Big(\QR{\mathcal{D}}{\ker(f)}\Big)}{\ker(f_n)}=\QR{\mathcal{D}}{\ker(f_n)}=\QR{\mathcal{D}}{\ker(\phi'_n)}=D_n\]
and the last equality holds because \(\phi_n':\mathcal{D}\rightarrow D_n\) is by definition a surjective \(^*\)-homomorphism.
Finally, the quotient map
\[\kappa_n:\mathcal{D}_{\infty}=\QR{\mathcal{D}}{\ker(f)}\longrightarrow\QR{\Big(\QR{\mathcal{D}}{\ker(f)}\Big)}{\ker(f_n)}=D_n\]
is a (norm-decreasing) \(^*\)-homomorphism.
Its extension \(\phi_n:D_{\infty}\rightarrow D_n\) is the desired arrow from \((D_{\infty},M_{\infty})\) to \((D_n,M_n)\) as it holds \(\phi_n(m^{(\infty)}_{ij})=m^{(n)}_{ij}\) by construction.
We conclude that a diagram as in Picture \ref{eqn:diagram_for_(D_infty,M_infty)} exists, so we can turn towards the universal property of \((D_{\infty},M_{\infty})\), described by Diagram \ref{eqn:universal_property_of_(D_infty,M_infty)}.
\newline
\textbf{Step 2: Universal property of \((D_{\infty},M_{\infty})\):}
Consider an object \((B,M)\) as described in Diagram \ref{eqn:diagram_for_(B,M)}.
Denote with \(\mathcal{B}\subseteq B\) the \(^*\)-subalgebra generated by the entries of \(M\).
By the definition of a limit we need to prove the existence of the commuting Diagrams \ref{eqn:universal_property_of_(D_infty,M_infty)}.
It suffices to prove existence of an arrow \(\psi\) from \((B,M)\) to \((D_{\infty},M_{\infty})\) as there is at most one arrow from one object to another.
To do so, we consider first a \(^*\)-algebraic expression \(b\) in the letters \(m_{ij}\) and we let \(\tilde{b}\) be the expression \(b\) but every letter \(m_{ij}\) is replaced by  \(m^{(\infty)}_{ij}\).
By the properties of our considered category, it holds
\[\psi_n(b)=\phi_n(\tilde{b})\]
for every \(n\in\N\).
As the \(\psi_n:B\rightarrow D_n\) are norm-decreasing, we deduce
\[\norm{b}_B\ge \sup_{n\in\N}\norm{\psi_n(b)}_{D_n}=\sup_{n\in\N}\norm{\phi_n(\tilde{b})}_{D_n}=\norm{\tilde{b}}_{D_{\infty}}.\]
i.e. the mapping \(m_{ij}\mapsto m^{(\infty)}_{ij}\) defines a norm-decreasing \(^*\)-homomorphism from \(\mathcal{B}\) to \(D_{\infty}\) and it can be extended to a \(^*\)-homomorphism \(\psi:B\rightarrow D_{\infty}\).
This finishes the proof of existence.
\newline
\textbf{Step 3: Uniqueness of \((D_{\infty},M_{\infty})\)} Uniqueness up to isomorphism is clear by the universal property of a limit. In the case of two limit objects we could switch roles to construct invertible arrows between them.
\end{proof}
\subsection{The inverse limit \((B_{\infty},M_{\infty})\)}
Considering the sequence of models \(\big((B_n,M_n)\big)_{n\in\N}\) as constructed in Section \ref{sec:a_sequence_of_models}, we have an inverse system
\[
(B_1,M_1)\overset{\pi_{1,2}}{\xlongleftarrow{\quad\quad}}\cdots \overset{\pi_{n-1,n}}{\xlongleftarrow{\quad\quad}}(B_n,M_n) \overset{\pi_{n,n+1}}{\xlongleftarrow{\quad\quad}}(B_{n+1},M_{n+1})\overset{\pi_{n+1,n+2}}{\xlongleftarrow{\quad\quad}}\cdots\quad.
\]
Its inverse limit
\[\lim_{\infty\leftarrow n}(B_n,M_n)=:(B_{\infty},M_{\infty})\]
exists by Lemma \ref{lem:existence_of_inverse_limits}.
The matrix \(M_1\) is a magic unitary as it defines a model of \(C(S_4^+)\), see Lemma \ref{lem:proving_all_properties_of_(B_1,M_1)}.
By Lemma \ref{lem:operp-product_gives_again_magic_unitary}, all matrices \(M_n\) are magic unitaries and so does \(M_{\infty}\), hence the inverse limit above defines a model of \(C(S_N^+)\),
\[\big(C(S_N^+),u_{S_N^+}\big)\overset{\varphi_{\infty}}{\xrightarrow{\quad\quad}}(B_{\infty},M_{\infty}).\]
It is larger than all models \((B_n,M_n)\) in the sense that we have, by definition of the limit of a diagram, arrows from \((B_{\infty},M_{\infty})\) to every \((B_n,M_n)\).
\vspace{11pt}
\newline
In this section we prove that this inverse limit is a compact matrix quantum group.
It only remains to show that on \(B_{\infty}\) there exists a comultiplication \(\Delta\) that fulfils
\[\Delta(m^{(\infty)}_{ij})=\sum_{k=1}^{N}m^{(\infty)}_{ik}\otimes m^{(\infty)}_{kj}.\]
In order to prove this, we consider the following situation.
Consider some \(N\in\N\) with \(N\ge 4\) and let \(P\) be a \(^*\)-polynom in the indeterminants \((X_{ij})_{1\le i,j\le N}\).
For \(n\in\N\cup\{\infty\}\) we define \(P(M_n)\) to be the element in \(B_{n}\) obtained by inserting the entries of \(M_n\) canonically into \(P\).
Analogously, let \(P(M_{n}\operp M_{n})\) be given by replacing \(X_{ij}\) with \(\displaystyle\sum_{k=1}^{N}m^{(n)}_{ik}\otimes m^{(n)}_{kj}\in B_{n}\otimes B_{n}\).
\newline
Existence of the comultiplication \(\Delta\) on \(B_{\infty}\) as described above is proved once we have shown the inequality
\[
\norm{P(M_{\infty})}_{B_{\infty}}\ge \norm{P(M_{\infty}\operp M_{\infty})}_{B_{\infty}\otimes B_{\infty}}\]
for all \(^*\)-polynomials \(P\).
\vspace{11pt}
\newline
The following results will be crucial in order to prove Theorem \ref{thm:B_infty_defines_CMQG}, saying that \((B_{\infty},M_{\infty})\) yields a CMQG.
The logical structure is as follows:
Lemma \ref{lem:linear_independence_is_reached_before_limit} is preparatory for Lemma \ref{lem:g_is_c*-norm} which in turn entails Lemma \ref{lem:g_equals_the_norm_on_B_infty_otimes_b_infty}.
Eventually, Lemma \ref{lem:g_equals_the_norm_on_B_infty_otimes_b_infty} and Lemma \ref{lem:(phi_n_otimes_phi_n)_circ_Delta_gives_entries_of_M_2n} are used in Theorem \ref{thm:B_infty_defines_CMQG}.
\begin{lem}\label{lem:linear_independence_is_reached_before_limit}
Consider the arrows
\[(B_{\infty},M_{\infty})\overset{\phi_k}{\longrightarrow}(B_k,M_k)\]
which exist by the property of an inverse limit.
Let \(a_1,\ldots,a_N\in B_{\infty}\) be linearly independent.
Then there is a \(K\in\N\) such that \(\phi_k(a_1),\ldots,\phi_k(a_N)\in B_k\) are linearly independent for all \(k\ge K\).
\newline
In particular, we find for any non-zero \(a_i\) some \(K\in N\) such that \(\phi_k(a_i)\!\neq\!0\) for all \(k\ge K\).
\end{lem}
\begin{proof}
Recall from the construction of an inverse limit, compare Lemma \ref{lem:existence_of_inverse_limits}, that the sequence of \(C^*\)-semi norms \(\big(\norm{\phi_n(\cdot)}_{B_n}\big)_{n\in\N}\) is increasing and its limit is the norm \(\norm{\cdot}_{B_{\infty}}\).
\newline
We now use induction on \(N\in\N\) to prove our claim.
For \(N=1\) we observe that a collection with only one element \(a_1\) is linearly independent if its element is non-zero, so we have \(0\neq \norm{a_1}_{B_\infty}=\lim_{k\rightarrow\infty}\norm{\phi_k(a_1)}_{B_k}\).
In particular \(\phi_k(a_1)\) is non-zero for all up to finitely many \(k\in\N\).
\newline
Now let the statement be proved for some \(N\in\N\) and consider linear independent elements \(a_1,\ldots,a_{N+1}\in B_{\infty}\).
We assume the opposite of our claim, i.e. we find arbitrary large \(k\in\N\) such that \(\phi_k(a_1),\dots,\phi_k(a_{N+1})\) are linearly dependent.
By the induction hypothesis we find \(K\in\N\) such that \(\phi_k(a_1),\ldots,\phi_k(a_N)\) are linearly independent for all \(k\ge K\).
So we find some \(L_1\ge K\) such that 
\[\phi_{L_1}(a_{N+1})=\sum_{i=1}^{N}\alpha_i\phi_{L_1}(a_i)\]
for suitable coefficients \(\alpha_i\).
As \(a_{N+1}-\sum_{i=1}^{N}\alpha_ia_i\) is  non-zero by linear independence of \(a_1,\ldots,a_{N+1}\), we find by the induction base case \(L_2\ge L_1\) such that
\[\phi_{l_2}(a_{N+1})\neq\sum_{i=1}^{N}\alpha_i\phi_{l_2}(a_i)\]
for all \(l_2\ge L_2\).
With the same arguments as before we find some \(L_3\ge L_2\ge K\) such that
\[\phi_{L_3}(a_{N+1})=\sum_{i=1}^{N}\beta_i\phi_{L_3}(a_i).\]
It holds \((\beta_1,\ldots,\beta_N)\neq(\alpha_1,\ldots,\alpha_N)\) because 
\[\phi_{L_3}(a_{N+1})\neq\sum_{i=1}^{N}\alpha_i\phi_{L_3}(a_i).\]
Defining \(\pi_{m,n}:=\pi_{n+1,n}\circ\ldots\circ\pi_{m,m-1}\), we conclude that
\begin{align*}
0&=\phi_{K}(a_{N+1})-\phi_K(a_{N+1})\\
&=\big(\pi_{L_3,K}\circ\phi_{L_3}\big)(a_{N+1})-\big(\pi_{L_1,K}\circ\phi_{L_1}\big)(a_{N+1})\\
&=\sum_{i=1}^{N}(\beta_i-\alpha_i)\phi_K(a_i),
\end{align*}
a contradiction to the linear independence of \(\phi_K(a_1),\ldots,\phi_K(a_N)\).
\end{proof}
\begin{lem}\label{lem:g_is_c*-norm}
The \(C^*\)-seminorm
\begin{equation}\label{eqn:defn_of_g}
g:=\sup_{n\rightarrow\infty}\norm{\big(\phi_n\otimes\phi_n\big)(\cdot)}_{B_n\otimes B_n}
\end{equation}
is a \(C^*\)-norm on the algebraic tensor product \(B_{\infty}\odot B_{\infty}\).
\end{lem}
\begin{proof}
Recall that the algebraic tensor product \(B_{\infty}\odot B_{\infty}\) is linearly spanned by elements \(x\otimes y\) with \(x,y\in B_{\infty}\).
For the proof we fix \(0\neq x=\sum_{i=1}^{N}a_i\otimes b_i\) with \(a_i,b_i\in B_{\infty}\), all \(b_i\neq 0\) and \(a_1,\ldots,a_N\) linearly independent.
Note that the sequence \(\big(\norm{\big(\phi_n\otimes\phi_n\big)(x)}_{B_n\otimes B_n}\big)_{n\in\N}\) is increasing, so the supremum in the statement is in fact a limit.
The statement is proved if we find an \(L\in\N\) such that  \(\norm{\big(\phi_L\otimes\phi_L\big)(x)}_{B_L\otimes B_L}\) is nonzero.
\newline
By Lemma \ref{lem:linear_independence_is_reached_before_limit} we find \(K\in\N\) such that for all \(k\ge K\) the elements \(\phi_k(a_1),\ldots,\phi_k(a_{N})\) are linearly independent.
As all \(b_i\) are non-zero, we find by Lemma \ref{lem:linear_independence_is_reached_before_limit} some \(L\ge K\) such that \(\phi_L(b_i)\neq 0\) for all \(1\le i\le N\).
But then we obviously have 
\[\sum_{i=1}^{N}\phi_L(a_i)\otimes \phi_L(b_i)\neq 0\]
as the first legs are linearly independent and the second ones are non-zero.
In particular, it holds 
\[\norm{\sum_{i=1}^{N}\phi_L(a_i)\otimes \phi_L(b_i)}_{B_L\otimes B_L}\neq 0.\]
\end{proof}
We even have that \(g\) defines a norm on \(B_ {\infty}\otimes B_{\infty}\) and it is equal to the norm on the minimal tensor product.
\begin{lem}\label{lem:g_equals_the_norm_on_B_infty_otimes_b_infty}
The mapping \(g\) from Lemma \ref{lem:g_is_c*-norm} is equal to the norm on \(B_{\infty}\otimes B_{\infty}\).
\end{lem}
\begin{proof}
Recall that the norm of a minimal tensor product \(\norm{\cdot}_{B\otimes C}\) of two \(C^*\)-algebras is by construction the smallest \(C^*\)-norm on \(B\odot C\) and it is defined by the supremum of the \(C^*\)-seminorms \(\norm{(\xi_1\otimes \xi_2)(\cdot)}_{B(H_1)\otimes B(H_2)}\) where \(\xi_1\) and \(\xi_2\) are representations of \(B\) on \(H_1\) and \(C\) on \(H_2\), respectively and \(\xi_1\otimes \xi_2\) is the product representation of \(B\odot C\) on \(H_1\otimes H_2\).
Furthermore, \(\xi_1\otimes \xi_2\) is faithful if both \(\xi_1\) and \(\xi_2\) are and in this case it holds \(\norm{\cdot}_{B\otimes C}=\norm{(\xi_1\otimes \xi_2)(\cdot)}_{B(H_1)\otimes B(H_2)}\).
\newline
It holds \(g\le\norm{\cdot}_{B_{\infty}\otimes B_{\infty}}\)because the \(C^*\)-semi norms \(\norm{(\phi_n\otimes \phi_n)(\cdot)}_{B_n\otimes B_n}\) all appear in the collection of semi norms \(\norm{(\xi_1\otimes \xi_2)(\cdot)}_{B(H_1)\otimes B(H_2)}\) as we can combine \(\phi_n\) with a faithful representation of \(B_n\).
\newline
Conversely, we have \(g\ge\norm{\cdot}_{B_{\infty}\otimes B_{\infty}}\) because \(g\) defines by Lemma \ref{lem:g_is_c*-norm} a \(C^*\)-norm on \(B_{\infty}\odot B_{\infty}\).
As \(\norm{\cdot}_{B_{\infty}\otimes B_{\infty}}\) is by definition the smallest possible \(C^*\)-norm on \(B_{\infty}\odot B_{\infty}\), we have \(g\ge \norm{\cdot}_{B_{\infty}\otimes B_{\infty}}\) 
\newline
Combing both inequalities, we conclude that \(g\) equals the minimal tensor product norm on \(B_{\infty}\odot B_{\infty}\), and therefore on the whole \(B_{\infty}\otimes B_{\infty}\).
\end{proof}
The following result is preparatory for Theorem \ref{thm:B_infty_defines_CMQG}.
\begin{lem}\label{lem:(phi_n_otimes_phi_n)_circ_Delta_gives_entries_of_M_2n}
For any \(n\in\N\) it holds
\begin{equation}\label{eqn:applying_phi_n_to_the_polynomial_p_a.s.o.}
\phi_{2n}\big(P(M_{\infty})\big)=(\phi_n\otimes\phi_n)\big(P(M_{\infty}\operp M_{\infty})\big)
\end{equation}
as an equation in \(B_1^{\otimes (2n)}=B_1^{\otimes n}\otimes B_1^{\otimes n}\) (or any suitable \(C^*\)-subalgebra).
\end{lem}
\begin{proof}
Starting with the \(^*\)-polynomial \(P\), we obtain the left side of Equation \ref{eqn:applying_phi_n_to_the_polynomial_p_a.s.o.} by replacing \(X_{ij}\) by \(m^{(2n)}_{ij}\) and the right side by replacing it by \(\displaystyle\sum_{k=1}^{N} m^{(n)}_{ik}\otimes m^{(n)}_{kj}\).
Equality of both sides follows from the associativity of the \(\operp\)-product which in turn follows from the associativity of the tensor product:
It holds for \(1\le i,j\le N\)
\newline
\scalebox{0.86}{
\parbox{150mm}{
\begin{align*}
m^{(2n)}_{ij}&=(M_1^{\raisebox{-1pt}{\operp} (2n)})_{ij}\\
&\sum_{t_1,\ldots,t_{2n-1}=1}^{4} m_{it_1}^{(1)}\otimes\ldots\otimes m_{t_{2n-1}j}^{(1)}\\
&=
\sum_{t_n=1}^{4}\left(\left(\sum_{t_1,\ldots t_{n-1}=1}^{4}m_{it_1}^{(1)}\otimes\ldots\otimes m_{t_{n-1}t_n}^{(1)}
\right)\otimes\left(
\sum_{t_{n+1},\ldots,t_{2n-1}=1}^{4}m_{t_nt_{n+1}}^{(1)}\otimes\ldots\otimes m_{t_{2n-1}j}^{(1)}\right)\right)\\
&=\sum_{k=1}^{4}\left(M_1^{\operp n}\right)_{ik}\otimes\left(M_1^{\operp n}\right)_{kj}\\
&=\sum_{k=1}^{4}m^{(n)}_{ik}\otimes m^{(n)}_{kj}.
\end{align*}
}
}

\end{proof}

\begin{thm}\label{thm:B_infty_defines_CMQG}
The \(C^*\)-algebra \(B_{\infty}\) together with its matrix of generators \(M_{\infty}=\big(m_{ij}^{(\infty)}\big)_{1\le i,j\le N}\) defines a compact matrix quantum group \(G=(B_{\infty},M_{\infty})\).
\end{thm}
\begin{proof}
As mentioned above, the only thing left to prove is the existence of a \(^*\)-homomorphism \(\Delta:B_{\infty}\rightarrow B_{\infty}\otimes B_{\infty}\) fulfilling 
\[\Delta(m_{ij}^{(\infty)})=\sum_{k=1}^{4}m_{ik}^{(\infty)}\otimes m_{kj}^{(\infty)}\]
and this can be guaranteed by proving the inequality
\begin{equation}\label{eqn:desired_inequality}
\norm{P(M_{\infty})}_{B_{\infty}}\ge \norm{P(M_{\infty}\operp M_{\infty})}_{B_{\infty}\otimes B_{\infty}}\end{equation}
for all \(^*\)-polynomials \(P\) as described above.
Due to the fact that the sequence \(\norm{\phi_n(\cdot)}_{B_n}\) is not only bounded but also increasing and its limit defines the norm \(\norm{\cdot}_{B_{\infty}}\), it holds
\[
\norm{P(M_{\infty})}_{B_{\infty}}
=\lim_{n\rightarrow\infty}\norm{\phi_{n}\big(P(M_{\infty})\big)}_{B_n}
=\lim_{n\rightarrow\infty}\norm{\phi_{2n}\big(P(M_{\infty})\big)}_{B_{2n}}.
\]
Using the inclusion of \(C^*\)-algebras \(B_{2n}\subseteq B_n\otimes B_n\) together with Lemma \ref{lem:(phi_n_otimes_phi_n)_circ_Delta_gives_entries_of_M_2n} and Lemma \ref{lem:g_equals_the_norm_on_B_infty_otimes_b_infty}, we conclude
\begin{align*}
\lim_{n\rightarrow\infty}\norm{\phi_{2n}\big(P(M_{\infty})\big)}_{B_{2n}}
&=\lim_{n\rightarrow\infty}\norm{\phi_{2n}\big(P(M_{\infty})\big)}_{B_n\otimes B_n} \\
&=\lim_{n\rightarrow\infty}\norm{\big(\phi_{n}\otimes \phi_n\big)\big(P(M_{\infty}\operp M_{\infty})\big)}_{B_n\otimes B_n}\\
&=g\big(P(M_{\infty}\operp M_{\infty})\big)\\
&= \norm{P(M_{\infty}\operp M_{\infty})}_{B_{\infty}\otimes B_{\infty}}
\end{align*}

Hence, Inequality \ref{eqn:desired_inequality} is true and \(G:=(B_{\infty},M_{\infty})\) is a compact matrix quantum group.
\end{proof}
The proof of Theorem \ref{thm:B_infty_defines_CMQG} even shows that the comultiplication \(\Delta\) on the compact matrix quantum group \(G=(B_{\infty},M_{\infty})\) is isometric.
%
%
%
Moreover,  the diagram
\[\big(C(S_N^+),u_{S_N^+}\big)\overset{\varphi_{\infty}}{\longrightarrow}\big(B_{\infty},M_{\infty}\big)\overset{\phi\,\circ\,\phi_1}{\xlongrightarrow{\quad\quad}}\big(C(S_N),u_{S_N}\big)\]
further shows that the constructed compact matrix quantum group lies in between the quantum permutation group and its classical analogue, compare Definition \ref{defn:subgrouprelation}.
\begin{cor}
The compact matrix quantum group \(G=(B_{\infty},M_{\infty})\) from Theorem \ref{thm:B_infty_defines_CMQG} fulfils
\[S_N\subsetneq G\subseteq S_N^+.\]
\end{cor}
\begin{proof}
The only thing left to show is the inequality \(S_N\neq G\).
As the \(C^*\)-algebra \(A\) from Equation \ref{eqn:defn_of_A} is non-commutative, so are \(B_1\) and \(B_{\infty}\), hence the arrow
\[(B_{\infty},M_{\infty})\overset{\phi\,\circ\,\phi_1}{\xlongrightarrow{\quad\quad}}\big(C(S_N),u_{S_N}\big)\]
cannot be invertible.
\end{proof}
It is a long standing conjecture (see for instance \cite{banicauniformquantumgroups}) that the inclusion \(S_N\subseteq S_N^+\) is maximal, i.e. there is no compact matrix quantum group strictly in between them.
\begin{conjecture}\label{conj:inclusion_S_N_S_N^+_maximal}
For all \(N\!\in\!\N_{\ge 4}\), there is no quantum group $G$ with \(S_N\subsetneq G\subsetneq S_N^+\).
\end{conjecture}
This has been proved in \cite{banicauniformquantumgroups} for the cases \(N=4\) and \(N=5\).
Exploiting this, we obtain the following result and question.
\begin{cor}
For \(N\!\in\!\{4,5\}\) the compact matrix quantum group \(G=(B_{\infty},M_{\infty})\) from Theorem \ref{thm:B_infty_defines_CMQG} equals \(S_N^+\).
\end{cor}
\begin{question}\label{conjecture:nothing_in_between_S_N_and_S_N^+}
Is \(G=(B_{\infty},M_{\infty})\) equal to \(S_N^+\) for every \(N\!\in\!\N_{\ge 4}\)?
This is trivial if Conjecture \ref{conj:inclusion_S_N_S_N^+_maximal} is true.
\end{question}

Moreover, we are wondering whether the inverse system $(B_n,M_n)$ is stationary at some point (we believe this is not the case). We phrase it as the following question.

\begin{question}\label{QuestZwei}
Are there polynomials $P_n$ in the generators $u_{ij}\in C(S_N^+)$ such that $\varphi_n(P_n)=0$ for $\varphi_n:C(S_N^+)\to B_n$, but $\varphi_{n+1}(P_n)\neq 0$?
\end{question}

We believe that such polynomials exist although we cannot prove it. Such a sequence $(P_n)_{n\in\N}$ would show that  none of the maps $\varphi_n$ is injective. Note that (at least for $N=4$ and $N=5$) the models $(B_n,M_n)$ approximate $C(S_N^+)$ completely, hence the $\operp$ operation applied on such simple matrices as in Definition \ref{defn:(B_1,M_1)} or Example \ref{ex:R_from-introduction} is powerful enough to reproduce $C(S_N^+)$ eventually. In the case that Question \ref{QuestZwei} is answered affirmatively, one can produce infinitely many mutually different quantum permutation matrices using the $\operp$ operation.
\section{Generalization to easy quantum groups}
Orthogonal easy quantum groups have been defined for the first time in \cite{banicaspeicherliberation} and they have been generalized in \cite{tarragoweberclassificationpartitions} and \cite{tarragoweberclassificationunitaryQGs} to unitary easy quantum groups.
This section is aimed for readers familiar with easy quantum groups and we refer to the references above for more details.
The notions of \emph{(two-coloured) partitions} and \emph{quantum group relations}, see below, are adopted from \cite{jungweber_PQS}.
\vspace{11pt}
\newline
The definition of easy quantum groups is based on Tannaka-Krein duality, see \cite{woronowicztannakakrein}, saying that there is a one-to-one correspondence between compact matrix quantum groups and their intertwiner spaces. 
To define an easy quantum group, one starts with a so-called category of two-coloured partitions (of sets) and associated to it a collection of intertwiner spaces (which defines a compact matrix quantum group).
\newline
In this work, however, we reduce the theory of easy quantum groups to a simple construction: Starting with (suitable) sets \(\Pi\) of partitions, one can associate to every partition \(p\in\Pi\) a collection of \emph{quantum group relations \(R_p(u)\)} on the canonical generators of a compact matrix quantum group \(C\big(G_N(\Pi)\big)\).
We finish this section by generalizing the result of the last section to arbitrary easy quantum groups, compare Proposition \ref{prop:generalization_to_easy_QGs}:
Given for \(N\in\N\) an easy quantum group \((G_N(\Pi),u_{G_N(\Pi)})\) -- which is an object in \(\mathcal{C}_N\) -- and an arrow
\[(G_N(\Pi),u_{G_N(\Pi)})\overset{\varphi_1}{\longrightarrow}(B_1,M_1)\]
the whole construction of the sequence of models \(\big((B_n,M_n)\big)_{n\in\N}\) as in Section \ref{sec:a_sequence_of_models} is possible.
If there is in addition an arrow
\[(B_1,M_1)\overset{\nu}{\longrightarrow}\big(\C,\mathds{1}_{M_N(\C)}\big)\]
the construction of the inverse system
\[
(B_1,M_1)\overset{\pi_{2,1}}{\xlongleftarrow{\quad\quad}}\cdots \overset{\pi_{n,n-1}}{\xlongleftarrow{\quad\quad}}(B_n,M_n) \overset{\pi_{n+1,n}}{\xlongleftarrow{\quad\quad}}(B_{n+1},M_{n+1})\overset{\pi_{n+2,n+1}}{\xlongleftarrow{\quad\quad}}\cdots\quad
\]
as in Section \ref{sec:a_sequence_of_models} and \ref{sec:the_limit_object} is possible and its inverse limit \((B_{\infty},M_{\infty})\) is well-defined and gives a compact matrix quantum subgroup of \(G_N(\Pi)\).
\begin{defn}
Consider \(N\!\in\!\N\) and let \(u\!=\!(u_{ij})_{1\le i,j\le N}\) be an \(N\!\times\!N\)-matrix of generators.
Given \(k,l\in\N_0\), we associate with a partition \(p\in\mathcal{P}(k,l)\) the \(^*\)-algebraic relations
\begin{equation}\label{eqn:quantum_group_relations_long_form_prelims}
\sum_{t\in[N]^k}\delta_p(t,\gamma')u_{t_1\gamma_1}^{\omega_1}\cdots u_{t_k\gamma_k}^{\omega_k}=\sum_{t'\in[N]^l}\delta_p(\gamma,t')u_{\gamma'_1t'_1}^{\omega'_1}\cdots u_{\gamma'_lt'_l}^{\omega_l}
\end{equation}
for all \(\gamma\in[N]^k\) and \(\gamma'\in[N]^l\) on the matrix entries \(u_{ij}\) and call them the \emph{quantum group relations} \(\mathcal{R}_p^{Gr}(u)\) associated to \(p\) and \(u\).
\end{defn}
We now state the definition of easy quantum groups as formulated in \cite{jungweber_PQS}.
\begin{defn}\label{defn:easy_QG}
Consider \(N\!\in\!\N\) and Let \(\Pi\) be a set of partitions such that it contains \(\{\paarpartwb, \paarpartbw, \baarpartwb, \baarpartbw\}\subseteq \Pi\).
Then we call the universal \(C^*\)-algebra
\[A:=C^*\big(\,(u_{ij})_{1\le i,j\le N}\,|\,\forall p\in\Pi:\textnormal{ The relations \(\mathcal{R}_p^{Gr}(u)\) hold.}\,\big)\]
the \emph{non-commutative functions over the easy quantum group \(G_N(\Pi)\)}.
Analogous to the definition of compact matrix quantum groups, we write \(A\!=\!C\big(G_N(\Pi)\big)\), \(u\!=\!u_{G_N(\Pi)}\) as well as \(G_N(\Pi)\!=\!(A,u)\!=\!\big(C(G_N(\Pi)),u_{G_N(\Pi)}\big)\).
\end{defn}
\begin{rem}\label{rem:rem_on_defn_of_easy_QGs}
Note, that the relations associated to the partitions\linebreak \(\{\paarpartwb, \paarpartbw, \baarpartwb, \baarpartbw\}\) exactly say that \(u\) and \(u^{(*)}\) are both unitary matrices, compare \cite{tarragoweberclassificationunitaryQGs}.
This guarantees that the universal \(C^*\)-algebra \(A\) is well-defined.
It can be checked by straightforward computation, see also Remark \ref{rem:comment_on_prop:generalization_to_easy_QGs} and the proof of Proposition \ref{prop:generalization_to_easy_QGs}, that a quantum group relation \(\mathcal{R}_p^{Gr}(u)\) is preserved by the symbolwise replacement
\[u_{ij}\overset{\Delta}{\longmapsto} \sum_{k=1}^{N}u_{ik}\otimes u_{kj}.\]
Hence, any easy quantum group is a compact matrix quantum group indeed.
\end{rem}
\begin{ex}
Starting with
\[\Pi=\{\paarpartwb, \paarpartbw, \baarpartwb, \baarpartbw,\singletonw,\,
\setlength{\unitlength}{0.5cm}
\begin{picture}(2.3,1.5)
\put(0,-0.3) {$\circ$}
\put(0.6,-0.3) {$\circ$}
\put(1.2,-0.3) {$\circ$}
\put(1.8,-0.3) {$\circ$}
\put(0.2,0.2){\line(0,1){0.5}}
\put(0.8,0.2){\line(0,1){0.5}}
\put(1.4,0.2){\line(0,1){0.5}}
\put(2.0,0.2){\line(0,1){0.5}}
\put(0.2,0.7){\line(1,0){1.8}}
\end{picture}
\vspace{4pt}
\,\}\]
the construction of \(G_N(\Pi)\) due to Definition \ref{defn:easy_QG}  gives the compact matrix quantum group \(S_N^+\) as defined in Definition \ref{defn:S_N^+}.
\end{ex}
\begin{prop}\label{prop:generalization_to_easy_QGs}
Consider \(N\in\N\) and a set of partitions
\[\Pi\supseteq\{\paarpartwb, \paarpartbw, \baarpartwb, \baarpartbw\}.\]
Whenever there exists in the category \(\mathcal{C}_N\) an object \((B_1,M_1)\) and arrows
\[(C(G_N(\Pi),u_{G_N(\Pi)})\overset{\varphi_1}{\longrightarrow}(B_1,M_1)\overset{\nu}{\longrightarrow}(\C,\mathds{1}_{M_N(\C)}),\]
the formula
\[M_{n+1}:=M_n\operp M_1\]
defines in \(\mathcal{C}_N\) an inverse system
\[
(B_1,M_1)\overset{\pi_{2,1}}{\xlongleftarrow{\quad\quad}}\cdots \overset{\pi_{n,n-1}}{\xlongleftarrow{\quad\quad}}(B_n,M_n) \overset{\pi_{n+1,n}}{\xlongleftarrow{\quad\quad}}(B_{n+1},M_{n+1})\overset{\pi_{n+2,n+1}}{\xlongleftarrow{\quad\quad}}\cdots\quad
\]
where \(\pi_{n+1,n}\) is the restriction of \(\id_{B_1}^{\otimes n}\otimes \nu\) to \(B_{n+1}\).
The inverse limit 
\[(B_{\infty},M_{\infty})=:\lim_{\infty\leftarrow n}(B_n,M_n)\]
exists and it defines a compact matrix quantum group \(G\subseteq G_N(\Pi)\).
\end{prop}
\begin{proof}
By the aforementioned considerations it only remains to show that every \((B_n,M_n)\) defines a model of \(C\big(G_N(\Pi)\big)\), i.e. there exists an arrow
\[(C\big(G_N(\Pi)\big),u_{G_N(\Pi)})\overset{\varphi_n}{\longrightarrow}(B_n,M_n).\]
By the universal property of \(C\big(G_N(\Pi)\big)\) we only have to show that the relations \(\big(\mathcal{R}_p^{Gr}(M_n)\big)_{p\in\Pi}\) hold.
We use induction on \(n\in\N\).
\newline
For \(n\!=\!1\) the statement is true by assumption, more precisely, by existence of \(\varphi_1\).
Now assume the claim to be proved for some \(n\in\N\).
To simplify notation, we write 
\[M_n=(x_{ij})_{1\le i,j\le N}\quad\quad\textnormal{and}\quad\quad M_1=(y_{ij})_{1\le i,j\le N}.\]
Recall that the quantum group relations \(\mathcal{R}_p^{Gr}(u)\) read as
\[
\sum_{t\in[N]^k}\delta_p(t,\gamma')u_{t_1\gamma_1}^{\omega_1}\cdots u_{t_k\gamma_k}^{\omega_k}=\sum_{t'\in[N]^l}\delta_p(\gamma,t')u_{\gamma'_1t'_1}^{\omega'_1}\cdots u_{\gamma'_lt'_l}^{\omega_l}
\]
for all \(\gamma\in[N]^k\) and \(\gamma'\in[N]^l\).
Using repeatedly that these relations are fulfilled for the \(x_{ij}\)'s and the \(y_{ij}\)'s, we can directly check the respective quantum group relations for the matrix
\[M_{n+1}=M_n\operp M_1=\left(\sum_{s=1}^{N}x_{is}\otimes y_{sj}\right)_{1\le i,j\le N}
.\]
Given \(\gamma\in[N]^k\) and \(\gamma'\in[N]^l\), we compute
{\allowdisplaybreaks
\begin{align*}
&\sum_{t\in[N]^k}\sum_{s\in[N]^k}\delta_p(t,\gamma')x_{t_1s_1}^{\omega_1}\cdots x_{t_ks_k}^{\omega_k}\otimes y_{s_1\gamma_1}^{\omega_1}\cdots y_{s_k\gamma_k}^{\omega_k}\\
=&
\sum_{s\in[N]^k}\left(\sum_{t\in[N]^k}\delta_p(t,\gamma')x_{t_1s_1}^{\omega_1}\cdots x_{t_ks_k}^{\omega_k}\right)\otimes y_{s_1\gamma_1}^{\omega_1}\cdots y_{s_k\gamma_k}^{\omega_k}\\
=&
\sum_{s\in[N]^k}\left(\sum_{t'\in[N]^l}\delta_p(s,t')x_{\gamma'_1t'_1}^{\omega'_1}\cdots x_{\gamma'_lt'_l}^{\omega'_l}\right)\otimes y_{s_1\gamma_1}^{\omega_1}\cdots y_{s_k\gamma_k}^{\omega_k}\\
=&
\sum_{t'\in[N]^l}x_{\gamma'_1t'_1}^{\omega'_1}\cdots x_{\gamma'_lt'_l}^{\omega'_l}\otimes \left(\sum_{s\in[N]^k}\delta_p(s,t')y_{s_1\gamma_1}^{\omega_1}\cdots y_{s_k\gamma_k}^{\omega_k}\right)\\
=&
\sum_{t'\in[N]^l}x_{\gamma'_1t'_1}^{\omega'_1}\cdots x_{\gamma'_lt'_l}^{\omega'_l}\otimes \left(\sum_{s'\in[N]^l}\delta_p(\gamma,s')y_{t'_1s'_1}^{\omega'_1}\cdots y_{t'_ls'_l}^{\omega'_l}\right)\\
\overset{s'\leftrightarrow t'}{=}\!\!&
\sum_{t'\in[N]^l}\sum_{s'\in[N]^l}\delta_p(\gamma,t')x_{\gamma'_1s'_1}^{\omega'_1}\cdots x_{\gamma'_ls'_l}^{\omega'_l}\otimes y_{s'_1t'_1}^{\omega'_1}\cdots y_{s'_lt'_l}^{\omega'_l}.
\end{align*}
}
These are the relations \(\mathcal{R}^{Gr}_p(M_{n+1})\) and the proof is finished.
\end{proof}
\begin{rem}\label{rem:comment_on_prop:generalization_to_easy_QGs}
The computation in the proof of Proposition \ref{prop:generalization_to_easy_QGs} is exactly the one that proves existence of the comultiplication  \(\Delta\) on \(C\big(G_N(\Pi)\big)\), compare item (ii) of Remark \ref{rem:rem_on_defn_of_easy_QGs}:
The relations $R_p^{Gr}(u)$ imply $R_p^{Gr}(u\operp u)$.
\end{rem}

Let us end this article with a number of questions. As formulated in Question \ref{conjecture:nothing_in_between_S_N_and_S_N^+}, we do not know whether or not the quantum group $(B_{\infty},M_{\infty})$ from Theorem \ref{thm:B_infty_defines_CMQG} coincides with $S_N^+$ for general $N$. If it doesn't, it would be a very interesting object to study and it would answer Conjecture \ref{conj:inclusion_S_N_S_N^+_maximal} to the negative. Moreover, we have no understanding on the dependence on the initial matrix $M_1$; recall that the construction of $(B_{\infty},M_{\infty})$ is completely determined by the matrix $M_1$ -- Roland Speicher raised the question, whether different matrices $M_1$ yield different quantum groups $(B_{\infty},M_{\infty})$. A positive answer would solve Question \ref{conjecture:nothing_in_between_S_N_and_S_N^+} and disprove Conjecture \ref{conj:inclusion_S_N_S_N^+_maximal}.

Moreover, Theorem \ref{thm:B_infty_defines_CMQG} shows, that the tensor product construction $\operp$ is powerful enough to eventually produce an interesting quantum group (or to reproduce $S_N^+$). In Question \ref{QuestZwei}, we ask whether one has to apply $\operp$ infinitely many times or whether a finite application is sufficient for obtaining the limit object. Related is a question that Adam Skalski pointed out to us: For $N=4$, one may think of our matrix $M_1$  as being obtained from $S_2^+ * S_2^+$ (note however, that $S_2^+=S_2$, of course). What happens if one applies a similar construction to $O_2^+ * O_2^+$ in order to produce quantum groups between $O_4$ and $O_4^+$? (Or more generally, using $G_n * G_m$ for obtaining quantum subgroups of $G_{n+m}$, where  $(G_n)_{n\in\N}$ is any series of quantum groups.) In particular, it would be interesting to see whether the half-liberated quantum group $O_n^*$ may be constructed in this way, and whether one may produce an example $O_n^*\subsetneq G\subsetneq O_n^+$ whose existence is also an open question. These questions are related to the concept of topological generation of quantum groups from certain quantum subgroups.

\bibliography{Models_of_SN+}
\bibliographystyle{alpha}

\end{document}